\documentclass[11pt,reqno]{amsart}
\usepackage{amsmath, color,cite}
\usepackage{amssymb}
\usepackage{amsfonts}
\usepackage{graphicx}
\usepackage{hyperref}

\allowdisplaybreaks

\newtheorem{theorem}{Theorem}[section]

\newtheorem{lemma} {Lemma}[section]
\newtheorem{proposition} {Proposition}[section]

\theoremstyle{definition}
\newtheorem{definition}{Definition}[section]
\newtheorem{remark} {Remark}[section]

\numberwithin{equation}{section}

\newcommand{\ba}{\begin{aligned}}
\newcommand{\ea}{\end{aligned}}
\newcommand{\be}{\begin{equation}}
\newcommand{\ee}{\end{equation}}
\newcommand{\bnn}{\begin{eqnarray*}}
\newcommand{\enn}{\end{eqnarray*}}
\newcommand{\n}{\rho}
\newcommand{\ep}{\epsilon}

\newcommand{\g}{\gamma}
\newcommand{\p}{\partial}

\renewcommand{\div}{{\rm div}}
\newcommand{\na}{\nabla}
\newcommand{\la}{\label}

\def\u{{\bf u}}

\def\o{\Omega}
\def\de{\delta}
\def\si{\sigma}
\def\ti{\tilde}
\def\r{\mathbb{R}^{3}}
\def\lap{\triangle}
\def\a{\alpha}

\def\th{\theta}
\begin{document}

\title[Weak solutions   for   stationary active liquid crystal flows]
{Weak solutions to the equations of  stationary compressible   flows in active liquid crystals} 

\author[Z. Liang]{Zhilei Liang}
\address{School of Economic Mathematics, Southwestern  University of Finance and Economics, Chengdu  611130,  China}
\email{zhilei0592@gmail.com}

\author[A. Majumdar]{Apala Majumdar}
\address{Department of Mathematics and Statistics, University of Strathclyde, G1 1XQ, UK.}
\email{apala.majumdar@strath.ac.uk}

\author[D. Wang]{Dehua Wang}
\address{Department of Mathematics,  University of Pittsburgh, Pittsburgh, PA 15260, USA}
\email{dwang@math.pitt.edu}

\author[Y. Wang]{Yixuan Wang}
\address{Department of Mathematics,  University of Pittsburgh, Pittsburgh, PA 15260, USA}
\email{yiw119@pitt.edu }


\date{\today}
\begin{abstract} 


The equations of stationary compressible flows of active liquid crystals are considered in a bounded three-dimensional domain. The system consists of the stationary Navier-Stokes equations coupled with the equation of Q-tensors and the equation of the active particles. The existence of weak solutions to the stationary problem is established through a two-level approximation scheme, compactness estimates and weak convergence arguments. 
Novel techniques are developed to overcome the difficulties  due to the lower  regularity of stationary solutions, a Moser-type iteration is used to deal with the  strong  coupling of active particles and fluids, and some weighted estimates on the energy functions are achieved so that the weak solutions can be constructed  for all values of the adiabatic exponent $\gamma>1$.

%

\end{abstract}
\maketitle

\section{Introduction}
Active hydrodynamics refer to dynamical systems that  are continuously driven out of equilibrium state by injected energy effects on small scales and exhibit collective phenomenon on a large scale, for example, bacterial colonies, motor proteins, and living cells \cite{live2,live1,mjr}.  
Active systems have natural analogies with nematic liquid crystals because   the particles exhibit a
  orientational ordering at a high concentration due to the collective motion.
 In comparison with  the  passive nematic liquid crystals,   the system of active hydrodynamics is usually unstable and has novel characteristics such as low Reynolds numbers  and very different spatial and temporal patterns \cite{gi,sai}.  We refer the readers to \cite{mjr,gi,sai,B-T-Y-2014,K-F-K-L-2008,P-K-O-2004,R-Y-2013} and their references for the physical background, applications and modeling of active hydrodynamics. 
 Theoretical studies on active liquid crystals are relatively new and have  attracted a lot of attention   in recent years. 
For example, the evolutionary incompressible flows of active liquid crystals were studied in \cite{chen0,lian} and the evolutionary  compressible flows were investigated in \cite{chen,QW2021}. 
In this paper we are concerned with  the  stationary compressible flows of  active liquid crystals, described by  the following equations    in a bounded domain $\mathcal{O}\subset \r:$
    \begin{equation}\label{1}
\left\{\ba &\div(\n \u)=0,\\
&\u\cdot \na c-\lap c=g_{1},\\
&\div  (\n \u\otimes \u)+\na  \n^{\g}
 - \div\left(\mathbb{S}_{ns}(\na \u)+\mathbb{S}_{1}(Q) +\mathbb{S}_{2}(c,Q)\right)=\n g_{2},\\
 &\u\cdot \na Q+Q\o-\o Q +c_{*}Q{\rm tr}(Q^{2})+\frac{(c-c_{*})}{2}Q-b\left(Q^{2}-\frac{1}{3}{\rm tr}(Q^{2})
 \mathbb{I}\right)-\lap Q=g_{3},\ea \right.
\end{equation}
where    $\n,\,c,\,\u$ denote  the total density,  the concentration of active particles,  and the   velocity field, respectively;   the nematic tensor order parameter $Q$ is a traceless and symmetric $3\times 3$ matrix,  $\n^{\g}$ is the pressure with adiabatic exponent $\g>1,$ and  the  functions   $g_{i}\,(i=1,2,3)$
    are given external force terms.  We denote the Navier-Stokes    stress tensor  by
     \be\la{c0} \mathbb{S}_{ns}(\na \u) =\mu\left(\na \u+(\na \u)^{\top}\right)+\lambda  \div \u\mathbb{I},\ee
where  $(\na \u)^{\top}$ denotes the  transpose of $\na \u$,    $\mathbb{I}$ is
 the identity matrix,  and  the constants  $\mu,\,\lambda$ are   viscous coefficients satisfying the following physical requirement:
\be\la{01}\mu>0,\quad \mu+3\lambda \ge0.\ee
In \eqref{1},  $\o=\frac12(\na \u-(\na \u)^{\top}),$ 
 and the  additional     stress tensors are:  
     \be\la{b0}\mathbb{S}_{1}(Q)=-\na Q \odot \na Q+\frac{1}{2}|\na Q|^{2}\mathbb{I}
     +\frac{1}{2}\left(1+\frac{c_{*}}{2}{\rm tr}(Q^{2}) \right){\rm tr}(Q^{2})\mathbb{I},\ee and
  \be\la{b1}\mathbb{S}_{2}(c,Q)=Q\lap Q-\lap Q Q+\sigma_{*}c^{2}Q,\ee
  where $c_{*}>0$ and $\sigma_{*}\in \mathbb{R}$ are given constants. 
  The corresponding  evolutionary equations of compressible active liquid crystal flows can be found in  \cite{chen}.
 The equations  \eqref{1} can be regarded as the stationary version of  the evolutionary equations in  \cite{chen}  through the time-discretization and play an important role  in the long-time behavior of active hydrodynamics. However,  the mathematical analysis of the stationary equations \eqref{1} remains open. 
The aim of this paper is to construct the weak solutions to the stationary equations \eqref{1} subject to the following  structural conditions:
\be\la{1a}\u=0,\,\,\,\frac{\p c}{\p n}=0,\,\,\,\frac{\p Q}{\p n}=0,\quad {\rm on} \,\,\,\p\mathcal{O},\ee
and
\be\la{0r} \int_{\mathcal{O}} \n(x)dx=m_{1}>0\quad {\rm and}\quad  \int_{\mathcal{O}} c(x)dx=m_{2}>0,\ee
where $n$ denotes the outward unit normal vector of the boundary $\partial\mathcal{O}$, $m_1$ and $m_2$ are given constants. 
We remark that  the two conditions on the total mass and  the    total active particles in \eqref{0r}    guarantee  that the density function $\n$ and the particle concentration $c$ are uniquely determined.
For the modeling and analysis of the $Q$-tensor systems of nematic liquid crystals we refer the readers to \cite{B-M-2010,B-E-1994,L-W-2014,HMP2017,XuX2022} and references therein.

We now introduce some notation that will be frequently used throughout this article.
   For  given symmetric matrices  $A=(a_{ij})_{3\times3}$ and $B=(b_{ij})_{3\times3}$,     denote       ${\rm tr}(AB)=A:B=\sum_{i,j=1}^{3}a_{ij}b_{ij}$,  
   ${\rm tr}(A^2)=|A|^{2},$ and
$S_{0}^{3}:=\{A=(a_{ij})_{3\times3}:\,\,a_{ij}=a_{ji},\,\,{\rm tr}(A)=0\}$.
 For two vectors $a,\,b\in \mathbb{R}^{3}$,  denote $a\cdot b=\sum_{i=1}^{3}a_{i}b_{i}$ and
$a\otimes b =(a_{i}b_{j})_{3\times3}.$   
   Denote the 
   Sobolev spaces (cf. \cite{ad})  by
 \bnn\ba &\, W^{k,p}= W^{k,p}(\mathcal{O}),\quad L^{p}=W^{0,p},\quad H^{k}=W^{k,2},\quad p\in [1,\infty],\quad  k\in \mathbb{N}_{+}.
\ea\enn
Additionally,   we use         $W^{k, p}(\mathcal{O},\,\r)$ and $W^{k,p}(\mathcal{O},\,S_{0}^{3})$   for  the   Sobolev spaces
valued in  $\r$  and $S_{0}^{3}$, respectively.  
We denote by $|\mathcal{O}|$ the measure of the domain $\mathcal{O}$, and write $\int_{\mathcal{O}}f(x)dx$ as  $\int f$ for simplicity of notation.

We shall establish  the existence of weak solutions to the problem \eqref{1}-\eqref{0r} defined as follows.

\begin{definition}\la{defi}
The function $(\n,c,\u,Q)$ is  called   a weak solution to the boundary-value problem  \eqref{1}-\eqref{0r}  if there
is some exponent     $p>\frac{3}{2}$ such that
\bnn\ba &   \qquad  \qquad \n\ge0, \quad c\ge 0 \quad a.e.\,\,{\rm in}\,\,\mathcal{O},\\
& \n\in L^{p}(\mathcal{O}),\quad c\in  H^{2}(\mathcal{O}), \quad  \u\in H_{0}^{1}(\mathcal{O},\r),\quad Q\in H^{2}(\mathcal{O},S_{0}^{3}),\ea\enn 
satisfying the following properties:

(i). The equations \eqref{1} are satisfied in the sense of distributions,  \eqref{1a} holds true in the trace sense,   \eqref{0r}  holds true for given $m_{1}>0$ and $m_{2}>0;$

(ii).  $\eqref{1}_{1}$ is satisfied   in the sense of renormalized solutions, i.e.,
if $(\n,\u)$ is extended by zero outside $\mathcal{O}$, then
$$\div(b(\n) \u)+ \left(b'(\n)\n-b(\n)\right)\div \u=0,\quad \mathcal{D}'(\r),$$
where  $b\in C^{1}([0,\infty))$ with  $b'(z)=0$ if $z$ is  large,

(iii).  $\eqref{1}_{2}$ and $\eqref{1}_{4}$  are  satisfied   almost everywhere in $\mathcal{O}.$

\end{definition}

We are ready to state our main result.
\begin{theorem}\la{t}
Let   $\mathcal{O}\subset\mathbb{R}^3$ be  a bounded domain with   smooth boundary.  Assume that   the adiabatic exponent  $\g>1$,  the   constants $m_{1}>0$ and $m_{2}>0$,  and the  functions
   \be\la{0} g_{1}\in  L^{\infty}(\mathcal{O}), \quad g_{2} \in  L^{\infty}(\mathcal{O},\r), \quad g_{3} \in  L^{\infty}(\mathcal{O},S_{0}^{3})\ee
   are given.
   Then there exists a small constant $\mathfrak{m}_{2}$ that depends   on  $m_{1},$ $c_{*},$ $\si_{*},$ $\mu,\lambda,$ $\g,$      $\|g_{1}\|_{L^{\infty}},$ $\|g_{2}\|_{L^{\infty}},$ $\|g_{3}\|_{L^{\infty}}$ and  $|\mathcal{O}|$,
such that  if
 \be\la{s} m_{2}\in (0, \mathfrak{m}_{2}],\ee the problem \eqref{1}-\eqref{0r}  admits a  solution  $(\n,c,\u,Q)$    in the sense of Definition \ref{defi}.
\end{theorem}

\begin{remark} The smallness assumption  \eqref{s} is a technical condition  that is mainly used to overcome the strong nonlinearity caused by the concentration $c$ of  active particles.  
\end{remark}
\begin{remark}
In fact,   Theorem \ref{t} still holds true in  the case when  $c$  is any positive constant (hence $m_2=c|\mathcal{O}|$). 
\end{remark}

We shall prove  Theorem \ref{t} by constructing approximate  solutions  and  a  two-level limiting procedure. The approximate  solutions  are constructed  in light of   time-discretization technique from the    evolutionary equations   in   \cite{chen},  and the limits are based on  standard compactness theories developed in \cite{chen,fei,lion,novo}.   However, new difficulties arise   due to the lower  regularity of stationary solutions,  strong nonlinearity and complex coupling of active particles and fluids. 
In order to make our ideas  clear we   comment on our approach and novelty below. 

We begin  in Section 2 with  suitable  linear equations  to construct the approximations of system \eqref{1}.
  For a given function  $v$ in  the set  $  \{v\in W^{1,\infty}(\mathcal{O},\r),\,\,\,v=0\,\,\,{\rm on}\,\,\p \mathcal{O}\}$, we impose    the transport  equation $\eqref{1}_{1}$  with the extra diffusion $\ep^{2}\lap \n$ and obtain  $\n=\n[v,\ep]$ in Lemma \ref{lem2.1a}. With  the   force term  $g_{1}$ given in $\eqref{1}_{2}$ we can solve $c=c[v,\ep]$.
  Having  $\n=\n[v,\ep])$ and $c=c[v,\ep]$ in hand,  
  for a given $v$ and a given   function $\ti{Q}$ in the set   $\{\ti{Q}\in W^{2,\infty}(\mathcal{O},\r),\,\,\, \frac{\p \ti{Q}}{\p n}=0\,\,\,{\rm on}\,\,\p \mathcal{O}\}$
  we  are able to construct  the solution to a linear system of $Q$ in terms of  $v$ and $\ti{Q}$.
  In the same manner,  we   consider the  approximate  momentum  equations \eqref{n3b} and solve   $\u=\u[v,\ti{Q},\ep]$. We should point out that  the appearance of highest derivative of $Q$  due to \eqref{b1} requires   $W^{3,p}$ regularity for $Q$. 
  Moreover,   since  $\na v$ and $g_{3}$ are only in $L^{\infty}$,  we adopt the ideas in \cite{liang} and use   a  global  mollification technique  such that the above approximation   is smooth.

The approximate equations \eqref{n1} come  from the linear equations   in Section 2  and will be solved using  the  Schaefer Fixed Point Theorem (cf. \cite{evans}).  The approximate solutions are constructed by a two-level approximation scheme involving the artificial viscosity  and     artificial pressure.
However,   the strong nonlinearity in    the   quantities \bnn  \int c^{2}Q:\na \u \quad {\rm and}\quad \int c\u\cdot \na c\enn
causes  new difficulties in closing  the basic  {\it a priori} estimates. To this end,  we explore a   Moser-type iteration  such that $\|c\|_{L^{\infty}}$ can be bounded by  $\|c\|_{L^{1}}=m_{2}$, and  hence we are able to control   the  above mentioned   nonlinear terms provided that some     small assumption  on $m_{2}$ is made.   In this connection, we are allowed    to close  the energy  estimates to obtain the existence of approximate solutions, and further  improve the regularity  of the solutions  by a bootstrap argument. 

Next we shall take the limit in the approximate solutions as  first $\ep\to 0$ and then $\de\to 0$ through the weak convergence arguments.
We remark that the nonlinear coupling of   $c$ and $Q$ in the momentum equation   makes the limiting process  much more subtle.    
For example, for   the  integral quantity  \bnn\int \left(Q_{\ep}^{ik} \lap Q_{\ep}^{kj} -\lap Q_{\ep}^{ik} Q_{\ep}^{kj}
 \right) \p_{j}\phi \p_{i}\lap^{-1}(\n_{\ep}),\enn
 the $\ep$-limit is not obvious   because both $\lap Q_{\ep}^{kj}$ and $\p_{i}\lap^{-1}(\n_{\ep})$ are  only  weakly  convergent.  Fortunately, we can overcome the difficulty    using the  integration by parts  as well as the symmetry of $Q$; see \eqref{k4}    for a detailed explanation.

One disadvantage for  the   stationary problem is that it has  no  useful information on the density  other than $\|\n\|_{L^{1}},$ which  is very different from the evolutionary equations for which the higher regularity $\|\n\|_{L^{\g}}$ with $\g>1$ is available.  As a consequence we have extra difficulties   in taking $\de$-limit procedure (especially  if $\g>1$ is close to 1). Taking account of the  ideas   in \cite{freh,jz,lw,lw1},   
we use the  refined  weighted estimates on both pressure and kinetic energy functions. However, the involvement of  $c$ and $Q$ makes the proof much more complex  and delicate. We utilize different weighted functions in dealing with the boundary case and interior case, and  finally succeed   in obtaining  the uniform estimates   for all  adiabatic exponent $\g>1$ under the smallness assumption \eqref{s}.     This is  different  from  our previous papers \cite{lw,lw1} for Cahn-Hilliard/Navier-Stokes equations where the restriction $\g>\frac{4}{3}$ seems to be critical  because  the pressure   depends both on the density  and the concentration.  Once the  Proposition \ref{p2} is  obtained, we are able to use  the standard compactness theories in \cite{fei,lion}  to take $\de$-limit and complete the proof of Theorem \ref{t}.


The rest of paper is organized as follows. 
In Section 2, we introduce some linear equations and their preliminary existence results that will be used in the construction of approximate solutions.
In Section 3, we construct the approximate solutions by a two-level approximation scheme involving the artificial viscosity coefficient $\ep>0$ and   the parameter $\de>0$ in the artificial pressure, and prove the existence using the fixed point argument.
In Section 4, we take the limit as $\ep\to 0$ of the approximate solutions for any fixed $\de>0$,
and finally in  Section 5 we take the limit as $\de\to 0$ for the  vanishing of   the artificial pressure and conclude the existence of weak solutions.

\bigskip

\section{Preliminary Results on Linear Equations} 

In this section we present some 
linear equations, 
in preparation for constructing the approximate solutions to the problem \eqref{1}-\eqref{0r} in \eqref{n1} next section.

Define the following function spaces:
$$W_{0}^{1,\infty}(\mathcal{O},\r):=\{v\in W^{1,\infty}(\mathcal{O},\r),\,\,\,v=0\,\,\,{\rm on}\,\,\p \mathcal{O}\},$$
$$W_{n}^{2,\infty}(\mathcal{O},S_{0}^{3}):=\left\{Q\in W^{2,\infty}(\mathcal{O},S_{0}^{3}),\,\,\ \, \frac{\p Q}{\p n}=0\,\,{\rm on}\,\,\p\mathcal{O}\right\},$$
$$\mathcal{W}:=W_{0}^{1,\infty}(\mathcal{O},\r)\times W_{n}^{2,\infty}(\mathcal{O},S_{0}^{3}).$$
%
%
%
Now let   $\ep\in (0,1)$ and $p\in(1,\infty)$ be fixed. Recall $m_1$ and $m_2$ defined in \eqref{0r}.  The first lemma  is for the  solvability of  a relaxed transport equation with dissipation from \cite{novo}.

 \begin{lemma}{\rm \cite[Proposition 4.29]{novo}} \la{lem2.1a}
For any given $v\in  W_{0}^{1,\infty}(\mathcal{O},\r)$,   there exists  a  unique solution   $\n=\n[v]\in W^{2,p}(\mathcal{O})$  to the following problem
\be\la{s1}
\ep  \n + {\rm div}(\n v) =\ep^{2} \lap \n +\ep   \frac{m_{1}}{|\mathcal{O}|},\quad \frac{\p \n}{\p n}=0 \,\,\,{\rm on}\,\,\,\p \mathcal{O},
\ee
such that
  \be\la{a3}  \ep^{2}\int \na \n\cdot \na \eta -\int \n v\cdot\na \eta +\ep \int (\n-\n_{0})\eta=0, \quad \eta\in C^{\infty}(\overline{\mathcal{O}}).\ee
Moreover,
\be\la{a17} \n\ge0\,\,a.e.\,\,{\rm in}\,\,\mathcal{O},\quad\|\n\|_{L^{1}}=m_{1},\quad  \|\n\|_{W^{2,p}}\le C(\ep,p, m_{1}, \mathcal{O}, \|v\|_{W^{1,\infty}}).\ee
\end{lemma}

\begin{lemma}\la{lem2.1b}
For any given  $g_{1}\in L^{\infty}(\mathcal{O})$  and  $v\in  W_{0}^{1,\infty}(\mathcal{O},\r)$,   
  the following problem
 \be\la{s2}  v\cdot \na c  =  \lap c +g_{1},\quad \int c=m_{2},\quad \frac{\p c}{\p n}=0 \,\,\,{\rm on}\,\,\,\p \mathcal{O},\ee
has a unique nonnegative solution   $c=c[v]\in W^{2,p}(\mathcal{O})$. 
\end{lemma}

\begin{proof} Consider the  approximate equation
\be\la{s1a} \a c +v \cdot \na c=\lap c+g_{1} +\a \frac{m_{2}}{|\mathcal{O}|},\quad \a\in (0,1).\ee
  Following the proof of  \cite[Proposition 4.29]{novo},  we see that \eqref{s1a} has a solution $c_{\a}=c_{\a}(v)\in W^{2,p}(\mathcal{O})\,\,(1<p<\infty)$, satisfying
\bnn c_{\a}\ge0\,\,\,a.e.,\,\,\,  \mathcal{O},\quad \|c_{\a}\|_{L^{1}}=m_{2},\quad \|c_{\a}\|_{W^{2,p}}\le C(p, m_{2}, \mathcal{O}, \|v\|_{W^{1,\infty}}).\enn   
Then, taking  the limit  $\a\rightarrow 0^+$, we  complete  the proof.
\end{proof}

The following two lemmas can be obtained from the elliptic theory (see \cite{gt}).

\begin{lemma}\la{lem2.1c} 
For any given $(v,\ti{Q})\in  \mathcal{W}$, 
 the   following problem
 \begin{equation}\label{n3a}
\left\{\ba
&\lap Q-\ep Q=F^{1}(v,\ti{Q}):=v\cdot \na \tilde{Q} +\frac{(c-c_{*})}{2}\ti{Q}+c_{*}\ti{Q}{\rm tr}(\ti{Q}^{2})-b\left(\ti{Q}^{2}-\frac{1}{3}{\rm tr}(\ti{Q}^{2})\mathbb{I}\right)
\\
&\qquad \qquad \qquad\qquad \qquad\qquad +\ti{Q}\langle \tilde{\o}\rangle-\langle  \tilde{\o}\rangle\ti{Q}-\langle g_{3}\rangle,\\
& \,\, \frac{\p Q}{\p n}=0 \quad {\rm on}\,\,\,\p\mathcal{O},\ea \right.   
\end{equation}
has a unique solution $Q=Q[v,\ti{Q}]$ satisfying 
 \be\la{a17b} \|Q\|_{W^{3,p}} \le C \|F^{1}\|_{W^{1,p}}<\infty,\ee
 where   $c=c[v]$  is  solved  in    Lemma \ref{lem2.1b}, and  $\langle  \tilde{\o}\rangle=  \langle \na v\rangle -(\langle \na v\rangle)^{\top}$ with  $\langle f\rangle$ being a  smooth approximation of function $f$ globally  in $\mathcal{O}$.
\end{lemma}
\begin{remark}  We use the smooth approximations $\langle \o\rangle$ and $\langle g_{3}\rangle$  to guarantee that  $F^{1}$ belongs to $W^{1,p}$.
Such    approximations can be obtained through the global     mollification $\langle f\rangle=\eta_{\ep}*f$ with $\eta_{\ep}$   the    Friedrichs mollifier (see e.g.,  \cite{evans}). 
Due to the Neumann boundary condition, we impose
 $\ep Q$   to guarantee  that   the values of  function $Q$ is uniquely determined.
 \end{remark}

\begin{lemma}\la{lem2.1d} 
%
For any given   $(v,\ti{Q})\in \mathcal{W}$,  
  the   following problem   \begin{equation}\label{n3b}
\left\{\ba
&{\rm div} \mathbb{S}_{ns}(\na\u)=F^{2}(v,\ti{Q})
:=\ep  \n v+{\rm div} (\n v\otimes v)+\na (\de\n^{4}+\n^{\g})+\ep^{2}\na \n\cdot \na v \\
&\qquad\qquad\qquad\qquad  -\div\left(-\na Q\otimes \na Q+\frac{1}{2}|\na Q|^{2}\mathbb{I}
+\frac{1}{2}tr(Q^{2})\mathbb{I}+\frac{c_{*}}{4}(tr(Q^{2}))^{2}\mathbb{I}\right)\\
 &\qquad\qquad\qquad\qquad  -\div\left(Q\lap Q-\lap Q Q+\sigma_{*}c^{2}Q\right)-\n g_{2}\\
& \u=0,\quad {\rm on}\,\,\,\p\mathcal{O},\ea \right.
\end{equation}
has a unique solution $\u=\u[v,\ti{Q}]$ satisfying 
\be\la{a17c}\|\u\|_{W^{2,p}} \le C \|F^{2}\|_{L^{p}}<\infty.\ee
where both $\ep\in (0,1)$ and $\de\in (0,1)$ are fixed constants; $\n=\n[v]$, $c=c[v]$  and $Q=Q[v,\ti{Q}]$  are  determined in    Lemmas \ref{lem2.1a}-Lemma \ref{lem2.1c}.

\end{lemma}
  \begin{remark} The artificial pressure $\de\n^{4} $ is used to improve the integrability of density,   which will be used   in   subsequent analysis. \end{remark}

\bigskip

\section{Approximate Solutions} 

In this section we construct the approximate solutions to the problem  \eqref{1}-\eqref{0r}. 
Have the existence results for   the  lineaized problems    in Lemmas \ref{lem2.1a}-\ref{lem2.1d}, 
we consider the following nonlinear approximate system:
\begin{equation}\label{n1}
\left\{\ba & \ep  \n_{\ep} + {\rm div}(\n_{\ep} \u_{\ep}) =\ep^{2} \lap \n_{\ep} +\ep  \n_{0},\\
& \u_{\ep} \cdot \na c_{\ep}=\lap c_{\ep}+g_{1}, \\
& \div \mathbb{S}_{ns}(\na \u_{\ep})=F^{2}(\u_{\ep},Q_{\ep}),\\
& \lap Q_{\ep}=\ep Q_{\ep}+F^{1}(\u_{\ep},Q_{\ep}), \\
&\frac{\p \n_{\ep}} {\p n}=0,\quad \u_{\ep}=0,\quad   \frac{\p c_{\ep}} {\p n}=0,\quad \frac{\p Q_{\ep}} {\p n}=0 \quad {\rm on}\,\,\,\p\mathcal{O}, \\
&\int  \n_{\ep} =m_{1}, \quad   \int  c_{\ep} =m_{2}, \ea\right.
\end{equation}
where  $\n_{0}=\frac{m_{1}}{|\mathcal{O}|}$, 
 the  functions $F^{1}$ and $F^{2}$ are taken from  \eqref{n3a} and \eqref{n3b} respectively.

 The   theorem below states the existence of solutions to  problem \eqref{n1}.
\begin{theorem} \la{t3.1} 
Assume that \eqref{0} holds true and   $\ep$ is sufficiently small. Then there is a small constant  $\mathfrak{m}_{2}$  depending on $m_{1},$ $c_{*},$ $\mu,\lambda,$ $\g,$  $\ep,$ $\de,$  $|\mathcal{O}|$,    $\|g_{1}\|_{L^{\infty}},$ $\|g_{2}\|_{L^{\infty}},$ $\|g_{3}\|_{L^{\infty}},$  such that if 
$m_{2}\le \mathfrak{m}_{2}$,   the problem \eqref{n1} admits
  a solution $(\n_{\ep},c_{\ep},\u_{\ep},Q_{\ep})$ satisfying,    for  any  $p\in (1,\infty)$, 
 \be\la{a19}  0\le \n_{\ep}\in W^{2,p}(\mathcal{O}),\quad \|\n_{\ep}\|_{L^{1}(\mathcal{O})}=m_{1},\ee
 \be\la{a19a}  0\le c_{\ep}\in W^{2,p}(\mathcal{O}),\quad \|c_{\ep}\|_{L^{1}(\mathcal{O})}=m_{2}, \ee
\be\la{a20}\u_{\ep} \in  
W^{2,p}(\mathcal{O},\r),\quad  Q_{\ep}\in  W^{3,p}(\mathcal{O},S_{0}^{3}).\ee
\end{theorem}

\begin{proof} The proof is based on the    Schaefer  Fixed Point  Theorem (see, e.g., Chapter 9, Theorem 4 in  \cite{evans}).
Thanks to   Lemmas \ref{lem2.1a}-\ref{lem2.1d}, for   any given  $(v,\ti{Q})\in \mathcal{W}$, we have
  \be\la{nn1} (\u_{\ep},Q_{\ep})= A[v,\ti{Q}]:=(\u[v,\ti{Q}],Q[v,\ti{Q}]).\ee
By  \eqref{a17b} and \eqref{a17c}, it is clear   that   the operator    $A:\mathcal{W}\to \mathcal{W}$   is  compact. A straightforward computation shows    that   $A$  is continuous;  see, e.g., \cite{chen}. In order to apply the    Schaefer  Fixed Point  Theorem, we need to prove the following proposition:

\begin{proposition}\la{c} 
Assume that  $(\u_{\ep},Q_{\ep})$ is  a solution to the equations \eqref{n3a} and   \eqref{n3b}.   Then  the  set  
\be\la{a4} \left\{(\u_{\ep},Q_{\ep})\in \mathcal{W}   \left| \, \ba&(\u_{\ep},Q_{\ep})= t A[\u_{\ep},Q_{\ep}]\\
&   {\rm for\,\,some}\,\,\,t \in [0,1],\,\,{\rm and}\quad  \n_{\ep}=\n[\u_{\ep}],\,\,\, c_{\ep}=c[\u_{\ep}]
\ea\right.
\right\}\ee  is bounded.
\end{proposition}

From Proposition \ref{c},  we may use  the  {Schaefer}  Fixed Point  Theorem 
to conclude   that  $ (\u_{\ep},Q_{\ep})= A[\u_{\ep},Q_{\ep}]$ with $\n_{\ep}=\n[\u_{\ep}]$ and $c_{\ep}=c[\u_{\ep}]$.
This together with   Lemma \ref{lem2.1a} and   Lemma \ref{lem2.1b}  guarantee the existence of   the solution $(\n_{\ep},c_{\ep},\u_{\ep},Q_{\ep})$ to   the problem  \eqref{n1} for any fixed $\ep>0$.  Consequently,        the estimates \eqref{a19}-\eqref{a19a}  follow directly from \eqref{a17} and Lemma \ref{lem2.1b}.

We now prove Proposition \ref{c} as well as \eqref{a20}, leading to the  complete  proof of Theorem \ref{t3.1}.  We will drop the subscript $\ep$ and use $(\n,c,\u,Q)$ to denote $(\n_{\ep},c_{\ep},\u_{\ep},Q_{\ep})$  for the sake of simplicity.
Observe that  $(\n,c,\u,Q)$ solves  \begin{equation}\label{n4}
\left\{\ba &  \ep  \n + {\rm div}(\n \u) =\ep^{2} \lap \n +\ep \n_{0},\\
& \u \cdot \na c =\lap c+g_{1},  \\ 
& \lap Q= \ep Q+ t F^{1}(\u,Q),\\
&{\rm div} \mathbb{S}_{ns}(\na \u)=t F^{2}(\u,Q),\\
&\frac{\p \n}{\p n}=0,\,\,\,\,\u=0,  \,\,\,\frac{\p c}{\p n}=0,\,\,\,
\frac{\p Q}{\p n}=0,\quad {\rm on}\,\,\,\p\mathcal{O},\\
& \int \rho =m_1, \quad \int c =m_2.\ea \right.
\end{equation}
To prove Proposition  \ref{c},  it suffices to   show that there is  a   constant   $M<\infty$  independent of    $t$ such that
  \be\la{a8} \|(\u,Q)\|_{\mathcal{W}} <M.\ee


\subsection{Basic inequalities}
Multiplying  $\eqref{n4}_{1}$ by $\frac{t}{2}|\u|^{2}$ and  $\eqref{n4}_{4}$ by $\u$  respectively,  we get
  \be\la{1.1}\ba &\frac{\ep t}{2}\int  (\n+\n_{0})|\u|^{2}
   + t\int  \u\cdot \na \left(\de\n^{4}+\n^{\g} \right)  +\mu \int |\na \u|^{2}+(\lambda +\mu)\int |{\rm div} \u|^{2}\\
  & =t\int  \n g_{2}   \cdot \u-t\int \lap  Q  :(\u\cdot \na)   Q
+\frac{1}{2}  \div \u \, tr(Q^{2}) \left(1+\frac{c_{*}}{2}tr(Q^{2}) \right)\\
&\quad +t\int \div (Q\lap Q-\lap Q Q)\u-t\si_{*}\int c^{2}Q :\na \u,\ea\ee
where   we have  used \eqref{b0}   and the following computation
\bnn\ba &\int\div \left(-\na Q\odot \na Q+\frac{1}{2}|\na Q|^{2}\right) \u\\
&= -\int \p_{i}(\p_{i} Q^{kl}  \p_{j}  Q^{kl}) \u^{j}+\frac{1}{2}\int \u^{j}\p_{j}|\p_{i} Q|^{2}
  =-\int \lap  Q  :(\u\cdot \na)   Q. \ea\enn
By  $\eqref{n4}_{1},$ 
one deduces
\bnn\ba &\int  \u\cdot\na \left(\de\n^{4}+\n^{\g}\right)=\int \n \u\cdot\na \left( \frac{4\de}{3}\n^{3}+\frac{\g}{\g-1}\n^{\g-1}\right) \\
&=\ep \int \left(\frac{4\de}{3}\n^{3}+\frac{\g}{\g-1}\n^{\g-1}\right)(\n-\n_{0})
 +\ep^{2} \int \na \left(\frac{4\de}{3}\n^{3}+\frac{\g}{\g-1}\n^{\g-1}\right)\cdot\na \n\\
 &\ge \ep \int \left(\frac{\de}{3}\n^{4}+ \frac{1}{\g-1}\n^{\g}\right)-
 \ep \int \left(\frac{\de}{3}\n_{0}^{4}+\frac{1}{\g-1}\n_{0}^{\g}\right) + \ep^{2}\int \left( \de|\na \n^{2}|^{2}+\frac{4}{\g} |\na \n^{\frac{\g}{2}}|^{2}\right).
 \ea\enn
Then substituting the above estimate into    \eqref{1.1} gives
   \be\la{1.4}\ba & \frac{\ep t}{2}\int (\n+\n_{0})|\u|^{2}
 +\ep t\int\left(\frac{\de}{3}\n^{4}+ \frac{1}{\g-1}\n^{\g}\right)
 +\ep^{2}t\int \left( \de|\na \n^{2}|^{2}+\frac{4}{\g} |\na \n^{\frac{\g}{2}}|^{2}\right)\\
 &\quad +\mu \int |\na \u|^{2}+(\lambda +\mu)\int |{\rm div} \u|^{2}\\
   & \le  \ep t\int \left(\frac{\de}{3}\n_{0}^{4}+\frac{1}{\g-1}\n_{0}^{\g}\right)
   +t\int  \n g_{2}   \cdot \u\\
   &\quad -t\int \lap  Q  :(\u\cdot \na)   Q
+\frac{1}{2}  \div \u \, {\rm tr}(Q^{2}) \left(1+\frac{c_{*}}{2}{\rm tr}(Q^{2}) \right)\\
&\quad +t\int \div (Q\lap Q-\lap Q Q)\u-t\si_{*}\int c^{2}Q :\na \u\\
&=:\ep t\int \left(\frac{\de}{3}\n_{0}^{4}+\frac{1}{\g-1}\n_{0}^{\g}\right)+\sum_{i=1}^{4} I_{i}.\ea\ee

Next,   following \cite{chen} we  multiply    $\eqref{n4}_{3}$ by $-\lap Q+Q+c_{*}Q{\rm tr}(Q^{2})$  to obtain
 \be\la{1.3}\ba &\int  |\lap Q|^{2}+(1+\ep) \int  |\na Q|^{2}
 +\ep \int \left( |Q|^{2}+c_{*}|Q|^{4}\right)
 +t c_{*} \int \left(|Q|^{4}+c_{*}|Q|^{6}\right) \\
 &= 2t c_{*}\int \lap Q: Qtr(Q^{2})+t\int \frac{(c-c_{*})}{2}Q:(\lap Q-Q-c_{*}Qtr(Q^{2}))
 \\
 &\quad+t\int b\left(Q^{2}-\frac{1}{3}\mathbb{I}tr(Q^{2})\right):\left(-\lap Q+Q+c_{*}Qtr(Q^{2})\right)+
t\int (Q\langle \o\rangle-\langle \o\rangle Q):\lap Q\\
 &\quad-t\int (Q\langle \o\rangle-\langle \o\rangle Q):(Q+c_{*}Qtr(Q^{2}))+ t \int \u\cdot \na Q :\left(\lap Q-Q-c_{*}Qtr(Q^{2})\right)\\
 &\quad+t\int \langle g_{3}\rangle:\left(-\lap Q+Q+c_{*}Qtr(Q^{2})\right)\\
 &=:\sum_{j=1}^{7}J_{j}.\ea\ee

\subsection{Uniform in $\ep$ and $t$ estimates}  Now we estimate the terms on the right-hand side   in \eqref{1.4} and \eqref{1.3}.
In this subsection, the generic
 constant $C$ may rely    on   $ \lambda$, $\mu$, $ m_{1}$,
$\de$, $\g$,   $|\mathcal{O}|$, $c_{*}$, $\si_{*}$, $\|g_{1}\|_{L^{\infty}}$, $\|g_{2}\|_{L^{\infty}}$, $\|g_{3}\|_{L^{\infty}}$, but not on $t
$ and $\ep.$

Direct calculations show
\be\la{pp1} J_{1}=-2t c_{*}\int |\na Q|^{2}tr(Q^{2})-t c_{*}\int |\na tr(Q^{2})|^{2}\le0, \ee
 \be\la{pp2}\ba J_{6}&=t\int \u\cdot \na Q :(\lap Q-Q-c_{*}Qtr(Q^{2}))\\
&=t\int \u\cdot\na Q : \lap Q-
t \int  \u\cdot \na \left(\frac{1}{2}(trQ^{2})+\frac{c_{*}}{4}((trQ^{2}))^{2}\right) =-I_{2},\ea\ee
and by the fact that $Q$ is symmetric and $\o$ is skew-symmetric, one has
\be\la{pp1a}J_{5}=t\int (Q\o-\o Q):(Q+c_{*}Qtr(Q^{2}))=0.\ee
Moreover, we have the following computation:
\be\la{pp3}\ba I_{3}+J_{4}&=t\int \div (Q\lap Q-\lap Q Q)\u+t\int (Q\langle \o\rangle-\langle \o\rangle Q):\lap Q\\
&=t\int \div (Q\lap Q-\lap Q Q)\u+t \int (Q  \o - \o  Q):\lap Q\\
&\quad+t\int \left[(Q\langle \o\rangle-\langle \o\rangle Q)-(Q\langle \o\rangle-\langle \o\rangle Q)\right]:\lap Q \\
&=t\int \left[(Q\langle \o\rangle-\langle \o\rangle Q)-(Q\langle \o\rangle-\langle \o\rangle Q)\right]:\lap Q \\
&\le t C\|\langle \na \u\rangle-\na \u\|_{L^{2}}\|Q\|_{L^{\infty}}\|\lap Q\|_{L^{2}}\\
&\le t C\|\langle \na \u\rangle-\na \u\|_{L^{2}}\left(\|Q\|_{L^{4}}^{2}+\|\lap Q\|_{L^{2}}^{2}\right),\ea\ee
where the last equality   is  from \cite[Lemma A1]{chen0}, and the last inequality   is from the interpolation inequality.

As a result of \eqref{pp1}-\eqref{pp3}, inequalities  \eqref{1.4} and  \eqref{1.3} provide us
  \be\la{1.41}\ba &   \frac{\ep t}{2}\int (\n+\n_{0})|\u|^{2}
 +\ep t\int\left(\frac{\de}{3}\n^{4}+ \frac{1}{\g-1}\n^{\g}\right)
 +\ep^{2}t\int \left( \de|\na \n^{2}|^{2}+\frac{4}{\g} |\na \n^{\frac{\g}{2}}|^{2}\right)\\
 &\quad +\mu\int |\na \u|^{2}+ \int  |\lap Q|^{2}+  \int  |\na Q|^{2}
 +t c_{*} \int \left(|Q|^{4}+c_{*}|Q|^{6}\right) \\
  &\le C t +Ct \|\langle \na \u\rangle-\na \u\|_{L^{2}}\left(\|Q\|_{L^{4}}^{2}+\|\lap Q\|_{L^{2}}^{2}\right) +I_{1}+I_{4}+J_{2}+J_{3} +J_{7}.\ea\ee
 By   \eqref{0}, we have 
 \be\la{s11}\ba  &|I_{1}| +|J_{3}|+|J_{7}| \\
 &\le Ct\|g_{2}\|_{L^{\infty}}\|\n\|_{L^{\frac{6}{5}}}\|\u\|_{L^{6}} +C\si(1+  \|g_{3}\|_{L^{2}})\left(\|\lap Q\|_{L^{2}}+ \|Q\|_{L^{6}}^{5}+1\right)\\
 &\le C t+C t \|\n\|_{L^{\frac{6}{5}}}^{2} + \frac{t c_{*}^{2}}{4}
  \int  |Q|^{6} + \frac{\mu}{4}\int |\na \u|^{2}
 + \frac{1}{4}\int  |\lap Q|^{2},\ea\ee
 and \be\la{s11a}\ba  & |I_{4}|+|J_{2}| \\
 &\le C t \|c\|_{L^{\infty}}^{2}\|Q\|_{L^{2}}\|\na\u\|_{L^{2}}+Ct(1+ \|c\|_{L^{\infty}})
 \left(\|Q\|_{L^{2}}\|\lap Q \|_{L^{2}}+  \|Q\|_{L^{6}}^{4}+1\right)\\
 &\le Ct +Ct\|c\|_{L^{\infty}}^{6} + \frac{t c_{*}^{2}}{4}
  \int  |Q|^{6} + \frac{\mu}{4}\int |\na \u|^{2}
 + \frac{1}{4}\int  |\lap Q|^{2}.\ea\ee
Substituting  \eqref{s11} and \eqref{s11a} into  \eqref{1.41} leads to
  \be\la{a26n}\ba &  \ep  \int (\n+\n_{0})|\u|^{2}
 +\ep  \int\left(\frac{\de}{3}\n^{4}+ \frac{1}{\g-1}\n^{\g}\right)
 +\ep^{2} \int \left( \de|\na \n^{2}|^{2}+\frac{4}{\g} |\na \n^{\frac{\g}{2}}|^{2}\right)\\
 &\quad + \int |\na \u|^{2}+ \int  |\lap Q|^{2}+  \int  |\na Q|^{2}
 +  \int \left(|Q|^{4}+|Q|^{6}\right) \\
  &\le C  +C\|\n\|_{L^{\frac{6}{5}}}^{2}+C \|c\|_{L^{\infty}}^{6}+C \|\langle \na \u\rangle-\na \u\|_{L^{2}}\left(\|Q\|_{L^{4}}^{2}+\|\lap Q\|_{L^{2}}^{2}\right).\ea\ee
Observe that the  constant $C$ in \eqref{a26n} is independent of  $\ep$,   then we may  choose   $\ep$  sufficiently  small and use the standard properties of mollification such that   $C\|\langle \na \u\rangle-\na \u\|_{L^{2}}\le \frac{1}{2}$ to obtain 
  \be\la{a26}\ba &  \ep  \int (\n+\n_{0})|\u|^{2}
 +\ep  \int\left(\frac{\de}{3}\n^{4}+ \frac{1}{\g-1}\n^{\g}\right)
 +\ep^{2} \int \left( \de|\na \n^{2}|^{2}+\frac{4}{\g} |\na \n^{\frac{\g}{2}}|^{2}\right)\\
 &\quad + \int |\na \u|^{2}+ \int  |\lap Q|^{2}+  \int  |\na Q|^{2}
 +  \int \left(|Q|^{4}+|Q|^{6}\right) \\
  &\le C  +C\|\n\|_{L^{\frac{6}{5}}}^{2}+C \|c\|_{L^{\infty}}^{6}.\ea\ee

\subsection{$\ep$-dependent  regularity} Thanks  to  $\|\n\|_{L^{1}}=m_{1}$ and the interpolation inequalities, it follows from \eqref{a26} that
\be\la{1.41b}\ba &  \ep  \int (\n+\n_{0})|\u|^{2}
 +\ep  \int\left(\frac{\de}{3}\n^{4}+ \frac{1}{\g-1}\n^{\g}\right)
 +\ep^{2} \int \left( \de|\na \n^{2}|^{2}+\frac{4}{\g} |\na \n^{\frac{\g}{2}}|^{2}\right)\\
 &\quad + \int |\na \u|^{2}+ \int  |\lap Q|^{2}+  \int  |\na Q|^{2}
 +  \int \left(|Q|^{4}+|Q|^{6}\right) \\
  &\le C\left(1 + \|c\|_{L^{\infty}}^{6}\right),\ea\ee
   where and in the rest of this subsection, the constant  $C$ may   rely on $\ep.$

\bigskip

In order to bound $\|c\|_{L^{\infty}}$   in  \eqref{1.41b}, we need the following lemma:
  \begin{lemma}\la{mos} There exist   constants $C$ and $C_{1}$    depending  only on $|\mathcal{O}|$    such that
\be\la{3.3}\|c\|_{L^{\infty}}\le C(1+\|\u\|_{L^{6}})^{C_{1}}(1+\|g_{1}\|_{L^{\infty}})m_{2}.\ee
\end{lemma}
We will continue the proof of Theorem \ref{t3.1} and postpone the proof of Lemma \ref{mos}   to the end of this section.
With the help of  \eqref{3.3} and \eqref{0}, we estimate   \eqref{1.41b} as
 \be\la{a26av}\ba &  \ep  \int (\n+\n_{0})|\u|^{2}
 +\ep  \int\left(\frac{\de}{3}\n^{4}+ \frac{1}{\g-1}\n^{\g}\right)
 +\ep^{2} \int \left( \de|\na \n^{2}|^{2}+\frac{4}{\g} |\na \n^{\frac{\g}{2}}|^{2}\right)\\
 &\quad + \int |\na \u|^{2}+ \int  |\lap Q|^{2}+  \int  |\na Q|^{2}
 +  \int \left(|Q|^{4}+|Q|^{6}\right) \\
  &\le C \left(1 + \|c\|_{L^{\infty}}^{6}\right)\\
  &\le  C \left(1 +   m_{2}^{6}\|\na \u\|_{L^{2}}^{6C_{1}}\right)\\
  &\le 2 C,\ea\ee
where the last inequality   is valid if
\be\la{ss40} m_{2}\le (2C)^{-\frac{C_{1}}{2}}.\ee
We remark that, by  \eqref{ss40}, the choice of   $m_{2}$ depends only on $m_{1},$ $c_{*},$ $\si_{*},$ $\mu,\lambda,$ $\g,$ $\ep,$ $\de,$ $|\mathcal{O}|$,  $\|g_{1}\|_{L^{\infty}},$ $\|g_{2}\|_{L^{\infty}},$ $\|g_{3}\|_{L^{\infty}}.$   

Having \eqref{a26av} obtained,  we multiply  $\eqref{n4}_{2}$ by $c$ and utilize  $\eqref{0}$  to deduce
\be\ba\la{a26a1} \int |\na c|^{2}&\le \int |g_{1}c|+\int |\u\cdot \na c| |c|\\
&\le  \|c\|_{L^{\infty}}(\|g_{1}\|_{L^{\infty}}+\|\na c\|_{L^{2}}\|\na \u\|_{L^{2}})\\
&\le \frac{1}{2}\|\na c\|_{L^{2}}^{2}+C.\ea\ee
If we multiply $\eqref{n4}_{2}$  by $-\lap c$,  we obtain
\be\ba\la{a26a2} \int |\lap c|^{2}&\le \int |g_{1} \lap c|+|  \u \cdot \na c \lap c|\\
&\le C \|\lap c\|_{L^{2}}(1+\|\na c\|_{L^{3}}\|\na \u\|_{L^{2}})\\
&\le C \|\lap c\|_{L^{2}}(1+\|\na c\|_{L^{2}}^{\frac{1}{2}}\|\lap c\|_{L^{2}}^{\frac{1}{2}}\|\na \u\|_{L^{2}})\\
&\le  \frac{1}{2}\|\lap c\|_{L^{2}}^{2}+C\|\na c\|_{L^{2}}.\ea\ee
The last two estimates \eqref{a26a1} and \eqref{a26a2} guarantee that,
 for small $m_{2},$
\be\la{a26a} \|\na c\|_{L^{2}}^{2}+\|\lap c\|_{L^{2}}^{2} \le C. \ee

\bigskip

We next  consider the    Neumann boundary  problem:
\be\la{n}\lap \n= \div b \quad {\rm with} \quad \frac{\p \n} {\p n}\Big|_{\p \mathcal{O}}=0.\ee

\begin{lemma}{\rm \cite[Lemma 4.27]{novo} } 
\la{lem2.3}
 Let $p\in (1,\infty)$  and  $b\in  L^{p}(\mathcal{O},\r)$.  Then  the problem  \eqref{n} admits  a  solution  $\n\in W^{1,p}(\mathcal{O})$,   satisfying
\bnn \int \na \n\cdot \na \phi=\int b\cdot \na \phi,\quad \forall\,\,\,
\phi\in C^{\infty}(\overline{\mathcal{O}}),\enn
and  the  estimates
\bnn\|\na \n\|_{L^{p}}\le C(p,|\mathcal{O}|)\|b\|_{L^{p}}\quad{\rm and }\quad
\|\na \n\|_{W^{1,p}}\le C(p,|\mathcal{O}|)(\|b\|_{L^{p}}+ \|\div b\|_{L^{p}} ).\enn
\end{lemma}

\begin{lemma}{\rm \cite[Lemma 3.17]{novo}} \la{lem2.2g} There is a linear operator $\mathcal{B}= (\mathcal{B}^{1},\mathcal{B}^{2},\mathcal{B}^{3})$ which satisfies

\smallskip

{\rm (i)} Let $\overline{L^{p}}:=\left\{f\in L^{p}\,\,|\int f=0\right\}$ with $p\in (1,\infty).$ Then, \bnn  \mathcal{B}(f): \overline{L^{p}}\mapsto \left(W_{0}^{1,p}\right)^{3},\quad  \div \mathcal{B}(f)=f\,\,a.e.\,\,{\rm in}\,\,\mathcal{O},\quad \forall\,\,f\in\overline{L^{p}}.\enn

{\rm (ii)}  For any $g\in L^{p}(\mathcal{O},\r)$ with $g\cdot n|_{\p \mathcal{O}}=0$, \bnn \|\na \mathcal{B}(f)\|_{L^{p}}\le C \|f\|_{L^{p}},\quad \|\mathcal{B}(\div g)\|_{L^{p}}\le C\|g\|_{L^{p}},\enn where the constant $C$ depends only on $p$ and $|\mathcal{O}|.$
\end{lemma}

Rewrite   $\eqref{n4}_{1}$ as
\be\la{a30}  \ep^{2}\lap \n=\div (\n \u+\ep \mathcal{B}(\n-\n_{0})).\ee
Applying Lemma \ref{lem2.3}  to   \eqref{a30},   and  using \eqref{a26av},  Lemma \ref{lem2.2g},   we find
   \bnn\ba \|\na \n\|_{L^{4}}&\le C \|\n \u+  \mathcal{B}(\n-\n_{0})\|_{L^{4}} \\
  &\le C \|\n \u\|_{L^{4}}  + C \|\na \mathcal{B}(\n-\n_{0}) \|_{L^{4}} \\
  &\le C \|\u\|_{L^{6}} \|\n^{2}\|_{L^{6}}^{\frac{1}{2}}+C \|\n-\n_{0}\|_{L^{4}}\le   C,\ea\enn
then using $L^{p}$ estimate on \eqref{a30} yields
 \be\la{a33} \ba\|\n \|_{H^{2}}&\le C\|\div (\n \u+\ep \mathcal{B}(\n-\n_{0}))\|_{L^{2}}\\
 & \le  C\|\u\cdot\na \n +\n \div \u\|_{L^{2}}  +C\|\div \mathcal{B}(\n-\n_{0})\|_{L^{2}}  \le C.\ea\ee

By virtue of \eqref{a26av} and   \eqref{a26a}, one has  $\|\u \cdot \na Q +t F^{1}(\u,Q)\|_{W^{1,\frac{3}{2}}}  \le C$, and hence 
 \be\la{k1} \| Q\|_{W^{3,\frac{3}{2}}}  \le C\ee from  $\eqref{n4}_{4}$.
  By \eqref{a33} and \eqref{k1},  we deduce
$ \|t F^{2}(\u,Q)\|_{L^{\frac{3}{2}}}  \le C,$  which together with $L^{p}$ regularity and $\eqref{n4}_{3}$ imply    \be\la{k2}\|\u\|_{W^{2,\frac{3}{2}}}\le C.\ee Finally,  using  \eqref{a26a}, \eqref{k1}, \eqref{k2}, and      $L^{p}$  regularity,  we obtain from  $\eqref{n4}_{2}$ that
\be\la{k3}\| c\|_{W^{2,p}}\le C  \quad (p<6).\ee
As a result of \eqref{a33}-\eqref{k3},
using  bootstrap procedure generates,  for $p\in (1,\infty)$, 
\bnn\|(\u, c)\|_{W^{2,p}}\le C,\quad \|Q\|_{W^{3,p}}\le C.\enn 
We have completed the proof of  Proposition \ref{c} and \eqref{a20}, except that we still need to prove  Lemma   \ref{mos}.
\end{proof}

The last part of this section is to give a proof of  Lemma   \ref{mos}.
 \begin{proof}[Proof of  Lemma   \ref{mos}]  The proof of Lemma   \ref{mos} is based on  a Moser-type iteration technique.
Fix  $x_{0}\in \mathcal{O}.$   Let  $B_{R}=B_{R}(x_{0})\subset \mathcal{O}$ be a ball centered in $x_{0}$ with radius $R\le 1$, and  let  $\eta(x)$ be a smooth cut-off such that,  for all $\frac{R}{2}\le r<r{'}\le R$, 
 \bnn \eta(x) \equiv1\,\,{\rm if}\,\,\,x\in B_{r},\quad \eta(x)\equiv0\,\,{\rm if}\,\,\,x\notin B_{r{'}},\quad |\na \eta|\le \frac{2}{(r^{'}-r)}.\enn
In the sequel, we assume   $\|g_{1}\|_{L^{\infty}}\le \frac{1}{2}$. Otherwise, we will multiply    $\eqref{n4}_{2}$ by $(2\|g_{1}\|_{L^{\infty}})^{-1} $ and consider $\frac{c}{2\|g_{1}\|_{L^{\infty}}}.$

A simple computation shows
 \bnn\ba -\int \eta^{2} c^{p}\lap c &= \frac{4p}{(p+1)^{2}}\int \eta^{2}|\na c^{\frac{p+1}{2}}|^{2}
 +\frac{4}{p+1}\int \eta  c^{\frac{p+1}{2}} \na \eta \na c^{\frac{p+1}{2}} dx\\
 &\ge \frac{2p}{(p+1)^{2}}\int \eta^{2}|\na c^{\frac{p+1}{2}}|^{2}-\frac{2}{p}\int |\na \eta|^{2}  c^{p+1},  \ea\enn
and  \bnn\ba  \int \eta^{2} c^{p} \u\cdot \na c&=\frac{2}{p+1}\int \eta^{2}  \u \cdot \na c^{\frac{p+1}{2}}  c^{\frac{p+1}{2}}\\
&\le \frac{p}{(p+1)^{2}}\int \eta^{2}|\na c^{\frac{p+1}{2}}|^{2}+\frac{1}{p}\int   \eta^{2} |\u|^{2} c^{p+1}.\ea\enn
With the above two inequalities,  and the fact  that  $p^{2}\|g_{1}\|_{L^{\infty}}^{p+1}$ is uniformly bounded for any $p\ge1$,  we multiply $\eqref{n4}_{1}$ by  $\eta^{2} c^{p}\,\,(p\ge1)$ to obtain 
  \be\la{3.6a}\ba  \int \eta^{2}\left|\na c^{\frac{p+1}{2}}\right|^{2}dx
  &\le C\int   \left(|\na \eta|^{2} +\eta^{2} |\u|^{2}\right) c^{p+1}+Cp^{2}\int \eta^{2}c^{p} g_{1} \\
  &\le C\|\u\|_{L^{6}}^{2}\left(\int_{B_{r'}} c^{\frac{3(p+1)}{2}} \right)^{\frac{2}{3}}+C \int |\na \eta|^{2} c^{p+1}+C p^{2}\|g_{3}\|_{L^{\infty}}^{p+1}\\
   &\le C\|\u\|_{L^{6}}^{2}\left(\int_{B_{r'}}   c^{\frac{3(p+1)}{2}}\right)^{\frac{2}{3}}+C \int |\na \eta|^{2} c^{p+1}+C, \ea\ee where the constant $C$ may rely  on $R$ and  $|\mathcal{O}|$   but not on $p$.

Owing to  the Sobolev embeddings (cf. \cite{ad}),   for  $f\in H_{0}^{1}(B_{R})$  one has 
 $ \|f\|_{L^{6}}\le C\|\na f\|_{L^{2}}.$
Thus,
 \bnn\ba  \left(\int_{B_{r}}c^{3(p+1)}\right)^{\frac{1}{3}}&\le \left(\int\left|\eta c^{\frac{p+1}{2}}\right|^{6}\right)^{\frac{1}{3}}\\
 &\le C \int\left|\na(\eta c^{\frac{p+1}{2}})\right|^{2}\\
 &\le C \left(\int\eta^{2}|\na c^{\frac{p+1}{2}} |^{2}+\int c^{p+1}|\na\eta|^{2}\right), \ea\enn
which  together with  \eqref{3.6a} give us the following estimate:
  \be\la{3.7}\ba  \left(\int_{B_{r}}c^{3(p+1)}\right)^{\frac{1}{3}}
&\le C\|\u\|_{L^{6}}^{2}\left(\int_{B_{r'}}   c^{\frac{3(p+1)}{2}}\right)^{\frac{2}{3}}
+ C\int |\na \eta|^{2} c^{p+1}+C\\
&\le C\left(\|\u\|_{L^{6}}^{2}+\frac{1}{|r'-r|^{2}}\right)\left(\int_{B_{r'}}   c^{\frac{3(p+1)}{2}}\right)^{\frac{2}{3}}+C.\ea\ee
Choosing
  \bnn r'=r_{k-1} \quad {\rm and}\quad  r_{k}=\frac{R}{2}\left(1+\frac{1}{2^{k}}\right),\quad k=1,2,....,\enn
we obtain from   \eqref{3.7} that
 \be\la{3.8a}\ba   \left( \int_{B_{r_{k}}}c^{3(p+1)}\right)^{\frac{1}{2}}
 \le C(1+\|\u\|_{L^{6}}^{3}) 2^{3(k+1)}  \left( \int_{B_{r_{k-1}}}  c^{\frac{3(p+1)}{2}} \right)+C.\ea\ee

If   $ \int  c^{\frac{3(p+1)}{2}} $ is  bounded uniformly  in $p$, then   \eqref{3.3}  follows directly by taking $p\rightarrow \infty$, subject to a subsequence.
Otherwise,   $ \int  c^{\frac{3(p+1)}{2}} \rightarrow \infty$ as  $p$ goes to infinity.  Hence,  without loss of generality we may assume that  $ \int  c^{\frac{3(p+1)}{2}} \ge C$ for all $p\ge 1$ and rewrite    \eqref{3.8a}  as
 \be\la{3.8}\ba   \left( \int_{B_{r_{k}}}c^{3(p+1)}\right)^{\frac{1}{2}}
 \le C \left(1+\|\u\|_{L^{6}}^{3}\right) 2^{3(k+1)}  \left( \int_{B_{r_{k-1}}}  c^{\frac{3(p+1)}{2}} \right).\ea\ee
Selecting   $ \frac{3(p+1)}{2}=2^{k-1}$ in
 \eqref{3.8},  one has
  \bnn \left( \int_{B_{r_{k}}}c^{2^{k}}\right)^{\frac{1}{2}}
 \le C\left(1+\|\u\|_{L^{6}}^{3}\right) 2^{3(k+1)} \left(\frac{1}{R^{3}}\int_{r_{k-1}}c^{2^{k-1}}\right),\enn
 which yields  by the deduction argument
  \be\la{s6} \left( \int_{B_{r_{k}}}c^{2^{k}}\right)^{\frac{1}{2^{k}}}
 \le  C  \left(1+\|\u\|_{L^{6}}^{3}\right)^{a} 2^{b}
 \int_{B_{R}}c^{2} \le C\left(1+\|\u\|_{L^{6}}^{3}\right)^{a} \int_{B_{R}}c^{2},\ee
 where  \bnn a=\sum_{k=1}^{\infty}\frac{1}{2^{k}}<\infty,\quad b=\sum_{k=1}^{\infty}\frac{3(k+1)}{2^{k}}<\infty.\enn
Sending $k\rightarrow \infty$  in \eqref{s6}  yields  
\be\la{3.15}
\ba  \sup_{x \in B_{\frac{R}{2}}}c^{2}& =\lim_{k\rightarrow \infty}\left( \int_{B_{\frac{R}{2}}}c^{2^{k}}\right)^{\frac{1}{2^{k}}}  \le C\left(1+\|\u\|_{L^{6}}^{3}\right)^{a} \int_{B_{R}}c^{2}.\ea\ee
Then \eqref{3.3} follows from  \eqref{3.15} together  with the fact $c\ge0$ and $\|c\|_{L^{1}}=m_{2}$. 

We remark that for the  case of  boundary points, we  can apply  local   flattening technique since the domain has smooth boundary  $\p \mathcal{O}$; while in the  case when $x_{0}\in \mathcal{O}$  is near the   boundary, we follow similarly the   ideas in \cite[Section 4]{lw}. 
Therefore, we complete  the proof of Proposition \ref{c} as well as \eqref{a20} and hence the  proof of Theorem \ref{t3.1}. 

\end{proof}


\section{$\ep$-Limit   for  the Approximate Solutions}

In this section, we shall take the $\ep$-limit  of  the approximate solutions obtained in Theorem \ref{t3.1} as $\ep\to 0$ for fixed $\delta\in(0,1)$,   and  prove the existence of solutions to the following problem:

 \begin{theorem} \la{t4.1} Under the same   assumptions  as in Theorem \ref{t3.1},    the  system
 \begin{equation}\label{n6}
\left\{\ba &\div (\n \u)=0,\\
& \u\cdot \na c-\lap c=g_{1}, \\
&\div(\n \u\otimes \u)+\na \left(\de \n^{4} + \n^{\g} \right)
 -\div\left(\mathbb{S}_{ns}+ \mathbb{S}_{1}(Q)+\mathbb{S}_{2}(c,Q) \right)=\n  g_{2},\\
& \u\cdot \na Q+Q\o-\o Q +c_{*}Qtr(Q^{2})+\frac{(c-c_{*})}{2}Q-b\left(Q^{2}-\frac{1}{3}tr(Q^{2})
 \mathbb{I}\right)-\lap Q=g_{3},\\
 &\u=0,\,\,\,\frac{\p c}{\p n}=0,\,\,\,\frac{\p Q}{\p n}=0,\quad {\rm on} \,\,\,\p\mathcal{O},\ea \right.
\end{equation}    admits a solution  $(\n,c,\u,Q)$  in the sense of distributions for any $\delta\in(0,1)$, satisfying  \be\la{e7b}\int \n =m_{1},\,\, 0\le \n\in L^{5}(\mathcal{O}),\qquad \int c =m_{2},\,\, 0\le c\in H^{2}(\mathcal{O}),\ee
 \be\ba\la{e7c} \u\in  H_{0}^{1}(\mathcal{O},\r),\qquad Q\in H^{2}(\mathcal{O},S_{0}^{3}).\ea\ee In particular, $\eqref{n6}_{2}$ and $\eqref{n6}_{4}$ are satisfied  almost everywhere in $\mathcal{O},$ and
 $\eqref{n6}_{1}$   holds in the sense of renormalized solutions, namely,
\bnn   {\rm div}(b(\n) \u) +(b'(\n)\n-b(\n))\div \u=0\quad {\rm in}\quad \mathcal{D}'(\mathbb{R}^{3}),\enn
where $b(z)=z$, or  $b\in C^{1}([0,\infty))$ with $b'(z)=0$ for large $z$.
\end{theorem}

\begin{proof}  We shall establish  the uniform in $\ep$ estimates on the solutions   $(\n_{\ep}, c_{\ep},\u_{\ep},Q_{\ep})$ obtained in Theorem \ref{t3.1} and then take the limit as $\ep\to 0$.   We remark that the idea of the proof is in the spirit of the arguments for the steady Navier-Stokes equations; see, e.g., \cite{novo1, novo,lw}.
In this section,  the constants $C$ and $C_{1}$ are  generic  and  independent of $\ep.$

Firstly, it follows directly from \eqref{a26} that,  if   $1<\g\le  2,$
 \be\la{s9j}\ba&\|Q_{\ep}\|_{L^{4}}^{4}+ \|\na \u_{\ep} \|_{L^{2}}^{2}
 +\|\na Q_{\ep}\|_{L^{2}}^{2}+\|\lap Q_{\ep}\|_{L^{2}}^{2}+\ep^{2}\|\na \n_{\ep}\|_{L^{2}}^{2}\\
 &\le \|Q_{\ep}\|_{L^{4}}^{4}+ \|\na \u_{\ep} \|_{L^{2}}^{2}
 +\|\na Q_{\ep}\|_{L^{2}}^{2}+\|\lap Q_{\ep}\|_{L^{2}}^{2}+\ep^{2}\left(\|\na \n_{\ep}^{\frac{\g}{2}}\|_{L^{2}}^{2}+\|\na \n_{\ep}^{2}\|_{L^{2}}^{2}\right)\\
 &\le C+C\| \n_{\ep}\|_{L^{\frac{6}{5}}}^{2}+C\| c_{\ep}\|_{L^{\infty}}^{6};\ea\ee
while in the case of $\g>2$, we replace the  artificial pressure $\de \n_{\ep}^{4}$ in \eqref{n1} with $\de \n_{\ep}^{4}+\de \n_{\ep}^{2}$, and repeat  the deduction of   \eqref{a26}  to  conclude that \be\la{s9h}\ba&\|Q_{\ep}\|_{L^{4}}^{4}+ \|\na \u_{\ep} \|_{L^{2}}^{2}
 +\|\na Q_{\ep}\|_{L^{2}}^{2}+\|\lap Q_{\ep}\|_{L^{2}}^{2}+\ep^{2} \left(\|\na \n_{\ep}^{2}\|_{L^{2}}^{2}+\|\na \n_{\ep}^{\frac{\g}{2}}\|_{L^{2}}^{2}+\|\na \n_{\ep}\|_{L^{2}}^{2}\right)\\
 &\le C+C\| \n_{\ep}\|_{L^{\frac{6}{5}}}^{2}+C\| c_{\ep}\|_{L^{\infty}}^{6}.\ea\ee
From  \eqref{s9j} and  \eqref{s9h}  we conclude that,  for all $\g>1$, 
 \be\la{s9}\ba&\|Q_{\ep}\|_{L^{4}}^{4}+ \|\na \u_{\ep} \|_{L^{2}}^{2}
 +\|\na Q_{\ep}\|_{L^{2}}^{2}+\|\lap Q_{\ep}\|_{L^{2}}^{2}+\ep^{2}\|\na \n_{\ep}\|_{L^{2}}^{2}\\
 &\le C\left(1+\| \n_{\ep}\|_{L^{\frac{6}{5}}}^{2}+\| c_{\ep}\|_{L^{\infty}}^{6}\right).\ea\ee
It follows  from   \eqref{0},  \eqref{s9}, and \eqref{3.3} that
 \be\la{s8}\ba 
 \|c_{\ep}\|_{L^{\infty}}&\le C\left(1+\|\u\|_{L^{6}}\right)^{C_{1}}(1+\|g_{1}\|_{L^{\infty}})m_{2}\\
 &\le C\left(1+\| \n_{\ep}\|_{L^{\frac{6}{5}}}+ \| c_{\ep}\|_{L^{\infty}}^{3}\right)^{C_{1}}m_{2}\\
 &\le C+Cm_{2} \| \n_{\ep}\|_{L^{\frac{6}{5}}}^{C_{1}}+ Cm_{2} \| c_{\ep}\|_{L^{\infty}}^{3C_{1}}\\
 &\le 2C+Cm_{2} \| \n_{\ep}\|_{L^{\frac{6}{5}}}^{C_{1}},\ea\ee
 where the last inequality   is valid as long as   $m_{2}$ is chosen sufficiently small.

\begin{lemma}\la{lem4.1} Let    $(\n_{\ep}, c_{\ep},\u_{\ep},Q_{\ep})$ be  a solution    in Theorem \ref{t3.1}.  Then  
 \be\la{b6} \|\n_{\ep}^{5}+\n_{\ep}^{\g+1}\|_{L^{1}}  \le C,\ee
  provided that $m_{2}$ is  sufficiently small.
\end{lemma}

\begin{proof}   Let $\mathcal{B}$  be the {Bogovskii} operator (see Lemma \ref{lem2.2g}).  Multiply    $\eqref{n1}_{3}$   by  $\mathcal{B}(\n_{\ep}-\n_{0})$  to obtain
 \be\la{b5}\ba&\int\left(\de\n_{\ep}^{4}+\n_{\ep}^{\g} \right)\n_{\ep} \\
&= \int \left(\de\n_{\ep}^{4}+\n_{\ep}^{\g}\right)\n_{0}-\int \n_{\ep} g_{2} \cdot \mathcal{B}(\n_{\ep}-\n_{0})\\
&\quad + \ep \int\n_{\ep} \u_{\ep} \cdot \mathcal{B}(\n_{\ep}-\n_{0})+\ep^{2}\int  \na \n_{\ep} \cdot \na \u_{\ep}\mathcal{B}(\n_{\ep}-\n_{0})-\int \n_{\ep} \u_{\ep}\otimes \u_{\ep}:\na\mathcal{B}(\n_{\ep}-\n_{0})\\
&\quad+ \int \mu(\na \u_{\ep}+(\na \u_{\ep})^{\top}):\na \mathcal{B}(\n_{\ep}-\n_{0})+  \lambda \div \u_{\ep}\div \mathcal{B}(\n_{\ep}-\n_{0})\\
&\quad+ \int \left(\frac{1}{2}|\na Q_{\ep}|^{2}\mathbb{I}-\na Q_{\ep} \odot \na Q_{\ep}\right):\na \mathcal{B}(\n_{\ep}-\n_{0})\\
&\quad+ \frac{1}{2}\int tr(Q_{\ep}^{2}) \left(1+\frac{c_{*}}{2}tr(Q_{\ep}^{2}) \right) \div \mathcal{B}(\n_{\ep}-\n_{0})\\
&\quad+ \int \left(Q_{\ep}\lap Q_{\ep}-\lap Q_{\ep} Q_{\ep} \right):\na \mathcal{B}(\n_{\ep}-\n_{0})+ \int \sigma_{*}c_{\ep}^{2}Q_{\ep} :\na \mathcal{B}(\n_{\ep}-\n_{0})  \\
&=: \sum_{i=1}^{10}K_{i}.\ea\ee
Using  $\|\n\|_{L^{1}}=m_{1}$  and interpolation,   one has
\bnn\ba K_{1}+K_{2}
&\le  C+ \frac{1}{16}\int\left(\de\n_{\ep}^{5}+\n_{\ep}^{\g+1} \right) +C \|\n_{\ep}\|_{L^{\frac{6}{5}}} \|\na \mathcal{B}(\n_{\ep}-\n_{0})\|_{L^{2}}\\
&\le C+ \frac{1}{16}\int\left(\de\n_{\ep}^{5}+\n_{\ep}^{\g+1} \right) +C \|\n_{\ep}\|_{L^{\frac{6}{5}}} \| \n_{\ep}-\n_{0} \|_{L^{2}}\\
&\le C+ \frac{1}{8}\int\left(\de\n_{\ep}^{5}+\n_{\ep}^{\g+1} \right).\ea\enn
Thanks to   \eqref{s9} and \eqref{s8},
 \bnn\ba  K_{5}&=-\int \n_{\ep} \u_{\ep}\otimes \u_{\ep}:\na\mathcal{B}(\n_{\ep}-\n_{0})\\
 &\le\|\n_{\ep}\|_{L^{\frac{12}{5}}}\|\u_{\ep}\|_{L^{6}}^{2}\|\na\mathcal{B}(\n_{\ep}-\n_{0})\|_{L^{4}} \\&\le C\|\n_{\ep}\|_{L^{\frac{12}{5}}}\left(1+\|\n_{\ep}\|_{L^{\frac{6}{5}}}^{2}\right)
\|\na\mathcal{B}(\n_{\ep}-\n_{0})\|_{L^{4}}  \\
&\le \frac{\de}{8}\|\n_{\ep}\|_{L^{5}}^{5}+C.\ea\enn
In a similar way,  one deduces
 \bnn\ba  K_{3}+K_{4}+\sum_{i=6}^{9}K_{i}
 &\le  \left(\ep \|\n_{\ep}\|_{L^{2}}+\ep^{2}\|\na \n_{\ep}\|_{L^{2}}\right)\|  \u_{\ep}\|_{H^{1}}\|\mathcal{B}(\n_{\ep}-\n_{0})\|_{L^{\infty}}\\
 &\quad +C \left(\|\na \u_{\ep}\|_{L^{2}}+\|\na Q_{\ep}\|_{L^{4}}^{2}\right)\|\na\mathcal{B}(\n_{\ep}-\n_{0})\|_{L^{2}}\\
 &\quad
 +C\left(1+\|Q_{\ep}\|_{L^{6}}^{4}+\|\lap Q_{\ep}\|_{L^{2}}\|Q_{\ep}\|_{L^{6}}\right)\|\na\mathcal{B}(\n_{\ep}-\n_{0})\|_{L^{3}} \\
&\le C  \|\n_{\ep}\|_{L^{\frac{12}{5}}}^{2}\|\mathcal{B}(\n_{\ep}-\n_{0})\|_{W^{1,4}} \\
&\le \frac{\de}{8}\|\n_{\ep}\|_{L^{5}}^{5}+C.\ea\enn
Finally,  using \eqref{s9}  and \eqref{s8},  one deduces
\bnn\ba K_{10}&= \int \sigma_{*}c_{\ep}^{2}Q_{\ep} :\na \mathcal{B}(\n_{\ep}-\n_{0})\\
&\le \|c_{\ep}\|_{L^{\infty}}^{2}\|Q_{\ep}\|_{L^{6}} \|\na \mathcal{B}(\n_{\ep}-\n_{0})\|_{L^{\frac{6}{5}}}\\
&\le   Cm_{2}(1+\| \n_{\ep}\|_{L^{\frac{6}{5}}})^{C_{1}}\\
 &\le  C+Cm_{2}\| \n_{\ep}\|_{L^{5}}^{C_{1}}. \ea\enn
 Substituting the last three inequalities into \eqref{b5} and taking $m_{2}$ small, we get
 \be\la{b5d} \int\left(\de\n_{\ep}^{5}+\n_{\ep}^{\g+1} \right) \le C+Cm_{2} \| \n_{\ep}\|_{L^{5}}^{C_{1}}\le 2C.\ee
The   proof of Lemma \ref{lem4.1} is completed.
\end{proof}

With \eqref{b6}  obtained, we deduce from \eqref{s9} and \eqref{s8} that
 \be\la{s10}\| c_{\ep}\|_{L^{\infty}}+\|Q_{\ep}\|_{L^{4}}^{4}+ \|\na \u_{\ep} \|_{L^{2}}^{2}
 +\|\na Q_{\ep}\|_{L^{2}}^{2}+\|\lap Q_{\ep}\|_{L^{2}}^{2}+\ep^{2}\|\na \n_{\ep}\|_{L^{2}}^{2}\le C.\ee  
 Then   multiply     $\eqref{n4}_{2}$ firstly by $c_{\ep}$  and then by $-\lap c_{\ep}$ to deduce
\be\la{a26af}  \|\na c_{\ep}\|_{L^{2}} +\|\lap c_{\ep}\|_{L^{2}}  \le C.\ee
As a result of   \eqref{b6},  \eqref{s10},  and  \eqref{a26af}   we  can take $\ep$-limit of   $(\n_{\ep}, c_{\ep},\u_{\ep},Q_{\ep})$  subject to  some subsequence  so that, as $\ep\to 0$, 
\be\la{b12}  \n_{\ep} \rightharpoonup \n \,\,{\rm in}\,\, L^{5}\cap L^{\g+1},  \ee
\be\la{b10} (\na \u_{\ep},\,\na^{2} Q_{\ep},\na c_{\ep})\rightharpoonup  (\na \u,\,\na^{2} Q,\,\na c)\quad {\rm in}\,\,\,\, L^{2},\ee
\be\la{b11}  \u_{\ep}  \rightarrow   \u, \quad (Q_{\ep},\,c_{\ep}) \rightarrow    (Q,\,c)\,\,\,\,{\rm in}\quad W^{1,p}\,\,\,(1\le p<6), \ee
\be\la{b15}   \langle \na \u_{\ep}\rangle \rightarrow \na \u,\,\,\,\langle g_{3}\rangle \rightarrow g_{3}\quad{\rm in}\,\,\,\, L^{2},\quad \ee
\be\la{b15b}   \ep  \n_{\ep} \rightarrow 0,\,\,\,\,\,   \ep  \n_{\ep} \u_{\ep}\rightarrow 0,\,\,\,\,\, \ep^{2} \na \n_{\ep}\na \u_{\ep}\rightarrow 0,\,\,\,\,\ep^{2}  \na \n_{\ep} \rightarrow 0\,\,\,\,  {\rm in}\quad  L^{1}.\ee
and moreover,  it follows from  \eqref{b12} and  \eqref{b11}   that
\be\la{b14}\n_{\ep}^{4} \rightharpoonup \overline{\n^{4}}\quad {\rm in}\,\,\, L^{\frac{5}{4}},\quad \n_{\ep}^{\g} \rightharpoonup \overline{\n^{\g}}\,\,\, {\rm in}\,\,\, L^{\frac{\g+1}{\g}},\quad \n_{\ep}  \u_{\ep} \rightharpoonup  \n \u\,\,\,  {\rm in}\,\,\,  L^{2}, \ee
where and hereafter the weak limit of a function $f$ is denoted by $\overline{f}.$ 
Therefore,      with \eqref{b12}-\eqref{b14}  in hand, we are able  to   pass  the limit  as $\ep\to 0$ and obtain  the  following   equations in the weak sense:
 \begin{equation}\label{n6f}
\left\{\ba&\div (\n \u)=0,\\
& \u\cdot \na c-\lap c=g_{1}, \\
&\div(\n \u\otimes \u)+\na \left(\de \overline{\n^{4}} + \overline{\n^{\g}} \right)
 -\div\left(\mathbb{S}_{ns}+ \mathbb{S}_{1}+\mathbb{S}_{2} \right)=\n  g_{2},\\
& \u\cdot \na Q+Q\o-\o Q +c_{*}Qtr(Q^{2})+\frac{(c-c_{*})}{2}Q-b\left(Q^{2}-\frac{1}{3}tr(Q^{2})
 \mathbb{I}\right)-\lap Q=g_{3}.\ea \right.
\end{equation}  
In addition,   \eqref{e7b} and \eqref{e7c} follow  from \eqref{a19}, \eqref{a19a}, \eqref{b11}, and \eqref{4.20c}    below.  The   next  lemma shows that  $(\n,\u)$  is a    renormalized solution to $\eqref{n6f}_{1}$.

 \begin{lemma}\la{lem4.2} 
 Assume that  $(\n,\u)$ is  a weak solution to $\eqref{n6}_{1}$,  $\n\in L^{2}(\mathcal{O}) $ and $ \u\in  H^{1}_{0}(\mathcal{O},\r)$.  If  we extend   $(\n,\u)$   by zero outside $\mathcal{O},$  we have     \be\la{wq} {\rm div}(b(\n) \u) +(b'(\n)\n-b(\n))\div \u=0\quad {\rm in}\quad \mathcal{D}'(\mathbb{R}^{3}),\ee
where $b(z)=z$, or  $b\in C^{1}([0,\infty))$ with $b'(z)=0$ for large $z$.
  \end{lemma}
\begin{proof}The detailed proof is available in \cite[Lemma 2.1]{novo1}. 
\end{proof}

In order to   complete  the proof of  Theorem \ref{t4.1},  we need to verify  \be\la{b21} \overline{\n^{4}}=\n^{4},\quad \overline{\n^{\g}}=\n^{\g}.\ee
To this end,  let us define   \bnn\ba C^{1}([0,\infty))\ni b_{n}(\n) =\left\{\ba
&\n\ln (\n+\frac{1}{n}),\quad \quad \n\le n;\\
&(n+1)\ln (n+1+\frac{1}{n}),\,\,\n\ge n+1.\\
\ea \right.
\ea\enn
We see that   $b_{n}(\n)\rightarrow \n\ln\n$ a.e.    because of the fact: $\n  \in L^{1}$.   Select  $b_{n}$ in  \eqref{wq} and  send  $n\rightarrow \infty$   to  obtain
\bnn  {\rm div}(\u\n\ln \n)+\n \div \u =0\quad {\rm in}\quad \mathcal{D}'(\mathbb{R}^{3}).\enn This implies
\be\la{4.8}  \int  \n \div \u=0.
 \ee
On the other hand,   multiplying $\eqref{n1}_{1}$  by  $b_{n}'(\n_{\ep})$ gives
\be\la{4.9}\ba \int (b_{n}'(\n_{\ep})\n_{\ep}-b_{n}(\n_{\ep})) \div \u_{\ep}& =\ep
\int  \n_{0}b_{n}'(\n_{\ep})-\ep
\int  \n_{\ep} b_{n}'(\n_{\ep}) -\ep^{4}\int b_{n}''(\n_{\ep}) |\na  \n_{\ep}|^{2}\\
  &\le \ep\int  \n_{0}b_{n}'(\n_{\ep})-\ep \int  \n_{\ep}b_{n}'(\n_{\ep}). \ea \ee
Recalling    \eqref{b6} and the definition of $b_{n}$,   one deduces that
\bnn\ba& \lim_{n\rightarrow\infty} \int  \n_{0}b_{n}'(\n_{\ep})\\
&=\lim_{n\rightarrow\infty}\left( \int_{\{\n_{\ep}\le n\}}\n_{0}b_{n}'(\n_{\ep})+\int_{\{\n_{\ep}>n\}}\n_{0}b_{n}'(\n_{\ep}) \right)\\
&\le \lim_{n\rightarrow\infty} \int_{\{\n_{\ep}\le n\}} \n_{0}\left(\ln (\n_{\ep}+\frac{1}{n}) +\frac{\n_{\ep}}{\n_{\ep}+\frac{1}{n}}\right)+C\lim_{n\rightarrow\infty}{\rm meas}\,|\{x;\,\n_{\ep}\ge n\}|\\
&\le  \lim_{n\rightarrow\infty}\int_{\{1/2\le \n_{\ep}\le n\}} \n_{0} \ln (\n_{\ep}+\frac{1}{n})+\lim_{n\rightarrow\infty}\int \frac{\n_{0} \n_{\ep}}{\n_{\ep}+\frac{1}{n}} \\
&\le C.\ea\enn
 Similarly,  \bnn \lim_{n\rightarrow\infty} \int  \n_{\ep}b_{n}'(\n_{\ep})\le C.\enn   Therefore,  taking  sequentially $n\rightarrow\infty$ and   $\ep\rightarrow0$   in \eqref{4.9}, using   \eqref{4.8},
  \be\la{4.10}\int \overline{\n \div u} = \lim_{\ep\rightarrow0} \int  \n_{\ep} \div u_{\ep}\le 0=\int  \n \div u.\ee

Now  define   the following effective viscous flux:
\be\la{s13}\mathbb{F}_{\ep} =\de \n^{4}_{\ep}+\n_{\ep}^{\g}  -(2\mu+\lambda )\div \u_{\ep}\quad {\rm and}\quad   \overline{\mathbb{F}}=\de\overline{\n^{4}}+\overline{\n^{\g}}  -(2\mu+\lambda )\div \u.\ee

\begin{lemma}\la{lem4.3} Under the assumptions    in  Theorem \ref{t4.1}, the following property holds:
\be\la{4.6}\ba  \lim_{\ep\rightarrow0}\int \phi   \n_{\ep} \mathbb{F}_{\ep}=\int \phi \n \overline{\mathbb{F}},\quad \forall\,\,\,\phi\in C_{0}^{\infty}(\mathcal{O}).\ea \ee
\end{lemma}

Let us   continue to prove \eqref{b21}  with the  aid of \eqref{4.6}.  The proof  of Lemma \ref{lem4.3} is postponed to the end of this section.

  Sending   $ \phi \rightarrow 1$ in \eqref{4.6}, using \eqref{4.10} and \eqref{s13},  we get
 \be\la{4.11}\ba& \lim_{\ep\rightarrow0}\int \left(\de \n^{5}_{\ep}+
 \n_{\ep}^{\g+1}  \right)\le \int \n \left(\de\overline{\n^{4}}+\overline{\n^{\g}}  \right).\ea \ee
According to  \eqref{4.11},    we have
 \bnn\ba   \int \left(\de\overline{\n^{5}}+ \overline{\n^{\g+1}} \right) =\lim_{\ep\rightarrow0}\int \left(\de \n^{5}_{\ep}+
 \n_{\ep}^{\g+1}  \right) \le \int \n \left(\de\overline{\n^{4}}+\overline{\n^{\g} }\right),\ea \enn
which implies  \be\la{4.12a}\ba &\int  \de\left(\n\overline{\n^{4}}-  \overline{\n^{5}}\right)\ge \int \left(\overline{\n^{\g+1}}-\n\overline{\n^{\g}} \right) \ge   0,\ea \ee
where the last inequality is  due  to the  convexity.
Next,  for  given   constant $\beta>0$ and  $\eta\in C^{\infty}(\mathcal{O}),$
\bnn\ba
  0&\le \int  \left(\n_{\ep}^{4}-(\n+\beta\eta)^{4}\right)(\n_{\ep}-(\n+\beta\eta))\\
 &= \int  \left(\n_{\ep}^{5}-\n_{\ep}^{4}\n -\n_{\ep}^{4}\beta\eta-(\n+\beta\eta)^{4}\n_{\ep}+(\n+\beta\eta)^{5}\right).\ea \enn
By   \eqref{4.12a},  sending  $\ep\rightarrow0$ yields
\bnn\ba 0&\le \int \left(\overline{\n^{5}}-\n \overline{\n^{4}}-\overline{\n^{4}} \beta\eta+(\n+\beta\eta)^{4} \beta\eta\right)\le \int  \left(-\overline{\n^{4}} +(\n+\beta\eta)^{4} \right)\beta\eta.\ea\enn
Replacing  $-\beta$ with $\beta$ in  the  argument above,  and then sending   $\beta\rightarrow0,$  we get
      \bnn\ba   \int \left(\n^{4}  -\overline{\n^{4}}\right) \eta=0,\ea \enn
which  implies  $\overline{\n^{4}} =\n ^{4}$, and thus   $\n_{\ep}\rightarrow \n$  a.e.  in $\mathcal{O}$  due to the arbitrariness of   $\eta$, and hence for all  $s\in [1,5)$, from  \eqref{b12},     
\be\la{4.20c}   \n_{\ep}\rightarrow \n\quad  {\rm in}\quad L^{s}.\ee
As a result of     \eqref{4.20c} and \eqref{b12},    we obtain \eqref{b21} and thus complete the proof of  Theorem \ref{t4.1}. \end{proof}

It remains  to    prove Lemma \ref{lem4.3}.  

 \noindent {{\it Proof of Lemma \ref{lem4.3}.}}   Let  $\lap^{-1}(h)=K*h$ be  the convolution of $h$ with the fundamental solution $K$ of Laplacian in $\r$.
For convenience, we write $\eqref{n1}_{3}$ equivalently as
\be\la{n1f}\ba &\ep \n_{\ep}\u_{\ep}^{i}+\p_{j}(\n_{\ep} \u_{\ep}^{j}  \u_{\ep}^{i})+\p_{i} \mathbb{F}_{\ep}+\ep^{2} \na \n_{\ep}\cdot \na \u_{\ep}^{i}\\
&=  \n_{\ep} g_{2}^{i}+\mu\lap \u_{\ep}^{i}   \\
&\quad -\p_{j}\left(\p_{j}Q_{\ep} \p_{i}Q_{\ep}\right)+\frac{1}{2}\p_{i}|\na Q_{\ep}|^{2}+\frac{1}{2}\p_{i}\left(tr(Q_{\ep}^{2})(1+\frac{c_{_{*}}}{2}tr(Q_{\ep}^{2}))\right) \\ &\quad+\p_{j}\left(Q_{\ep}^{ik} \lap Q_{\ep}^{kj} -\lap Q_{\ep}^{ik} Q_{\ep}^{kj}
+\si_{*}c_{\ep}^{2}Q_{\ep}^{ij}\right),\quad i=1,2,3,\ea\ee
where   the Einstein summation is used on $k,j$, and $\mathbb{F}_{\ep}$  is  taken from  \eqref{s13}.

Making zero extension of  $\n_{\ep}$ to the whole space $\r$,
multiplying    $\eqref{n1f}$ by $\phi\p_{i}\lap^{-1}(\n_{\ep})$ with $\phi\in C_{0}^{\infty}(\mathcal{O})$,
we deduce
 \be\la{4.2}\ba  &\int  \phi \n_{\ep}\mathbb{F}_{\ep}\\
&=-\int \p_{i}\lap^{-1}(\n_{\ep}) \p_{i}\phi \left(\de\n_{\ep}^{4}+\n_{\ep}^{\g} -(\mu+\lambda )\div \u_{\ep}\right)-\int \n_{\ep} g_{2}^{i}\phi\p_{i}\lap^{-1}(\n_{\ep})\\
&\quad+\mu\int \left( \p_{j}\u^{i}_{\ep}\p_{i}\lap^{-1}(\n_{\ep}) \p_{j}\phi- \u^{i}_{\ep}\p_{j}\p_{i}\lap^{-1}(\n_{\ep})\p_{j}\phi+  \n_{\ep} \u_{\ep} \cdot\na \phi\right)\\
&\quad-\int \n_{\ep}\u_{\ep}^{j}\u_{\ep}^{i}\p_{j}\phi \p_{i}\lap^{-1}(\n_{\ep})-\int \n_{\ep}\u_{\ep}^{j}\u_{\ep}^{i}\phi\p_{j} \p_{i}\lap^{-1}(\n_{\ep})\\
&\quad + \int\left(\frac{1}{2}\int |\na Q_{\ep}|^{2}+\frac{1}{2}  tr(Q_{\ep}^{2})(1+\frac{c_{_{*}}}{2}tr(Q_{\ep}^{2}))\right)\left(\p_{i}\phi \p_{i}\lap^{-1}(\n_{\ep})+\phi \n_{\ep} \right)\\
&\quad -\int \p_{j}Q_{\ep} \p_{i}Q_{\ep} \left(\p_{j}\phi \p_{i}\lap^{-1}(\n_{\ep})+\phi\p_{j} \p_{i}\lap^{-1}(\n_{\ep})\right)\\
&\quad +\int \left(Q_{\ep}^{ik} \lap Q_{\ep}^{kj} -\lap Q_{\ep}^{ik} Q_{\ep}^{kj}
+\si_{*}c_{\ep}^{2}Q_{\ep}^{ij}\right)\left(\p_{j}\phi \p_{i}\lap^{-1}(\n_{\ep})+\phi \p_{j}  \p_{i}\lap^{-1}(\n_{\ep}) \right)\\
&\quad+\ep \int \n_{\ep}\u_{\ep}^{i} \phi  \p_{i}\lap^{-1}(\n_{\ep}) +\ep^{2}\int  \na \n_{\ep}\cdot \na \u_{\ep}^{i}\phi\p_{i}\lap^{-1}(\n_{\ep}),
\ea\ee
where the second line on the right-hand side  is due to
\bnn \ba
& \int  \p_{j}\u^{i}_{\ep}\left(\p_{i}\lap^{-1}(\n_{\ep})\p_{j}\phi+\p_{j}\p_{i}\lap^{-1}(\n_{\ep})\phi\right)\\
&= \int \left( \p_{j}\u^{i}_{\ep}\p_{i}\lap^{-1}(\n_{\ep}) \p_{j}\phi- \u^{i}_{\ep}\p_{j}\p_{i}\lap^{-1}(\n_{\ep})\p_{j}\phi+  \n_{\ep} \u_{\ep} \cdot\na \phi\right)+ \int \n_{\ep} \div  \u_{\ep}\phi.\ea\enn
Making use of
\bnn \ep(\n_{\ep}-\n_{0})+\div (\n_{\ep}\u_{\ep})=\ep^{2}\div (\textbf{1}_{\mathcal{O}}\na \n_{\ep})\quad {\rm in}\quad \r,\enn
 we write the third line on the right-hand side of \eqref{4.2} as
\be\la{4.2q}\ba
&-\int  \n_{\ep}\u_{\ep}^{j}\u_{\ep}^{i}\p_{j}\phi \p_{i}\lap^{-1}(\n_{\ep})-\int \n_{\ep}\u_{\ep}^{j}\u_{\ep}^{i}\phi\p_{j} \p_{i}\lap^{-1}(\n_{\ep})\\
&=-\int \n_{\ep}\u_{\ep}^{j}\u_{\ep}^{i}\p_{j}\phi \p_{i}\lap^{-1}(\n_{\ep})
 +\int \u_{\ep}^{i}\phi \left[\n_{\ep}\p_{i}\p_{j}\lap^{-1}(\n_{\ep}\u^{j}_{\ep})-\n_{\ep}\u_{\ep}^{j} \p_{j}\p_{i}\lap^{-1}(\n_{\ep}) \right]\\
 &\quad-\int \n_{\ep}\u^{i}_{\ep}\phi\p_{i}\p_{j}\lap^{-1}(\n_{\ep}\u_{\ep}^{j} )\\
&=-\int  \n_{\ep}\u_{\ep}^{j}\u_{\ep}^{i}\p_{j}\phi \p_{i}\lap^{-1}(\n_{\ep}) +\int \u_{\ep}^{i}\phi \left[\n_{\ep}\p_{i}\p_{j}\lap^{-1}(\n_{\ep}\u^{j}_{\ep})-\n_{\ep}\u_{\ep}^{j} \p_{j}\p_{i}\lap^{-1}(\n_{\ep}) \right]\\
&\quad-\ep^{2}\int  \n_{\ep}\u^{i}_{\ep}\phi \p_{i}\lap^{-1}(\div(\textbf{1}_{\mathcal{O}} \na\n_{\ep}))+\ep\int \n_{\ep}\u^{i}_{\ep}\phi \p_{i}\lap^{-1}(\n_{\ep}-\n_{0}).\ea\ee
Substituting  \eqref{4.2q} into   \eqref{4.2} gives us
\be\la{4.3}\ba  &\int\phi \n_{\ep}\mathbb{F}_{\ep}  \\
&=-\int \p_{i}\lap^{-1}(\n_{\ep}) \p_{i}\phi \left(\de\n_{\ep}^{4}+\n_{\ep}^{\g} -(\mu+\lambda )\div \u_{\ep}\right)-\int \n_{\ep} g_{2}^{i}\phi\p_{i}\lap^{-1}(\n_{\ep})\\
&\quad+\mu\int \left( \p_{j}\u^{i}_{\ep}\p_{i}\lap^{-1}(\n_{\ep}) \p_{j}\phi- \u^{i}_{\ep}\p_{j}\p_{i}\lap^{-1}(\n_{\ep})\p_{j}\phi+  \n_{\ep} \u_{\ep} \cdot\na \phi\right)\\
&\quad-\int \n_{\ep}\u_{\ep}^{j}\u_{\ep}^{i}\p_{j}\phi \p_{i}\lap^{-1}(\n_{\ep})+\int \u_{\ep}^{i}\phi \left[\n_{\ep}\p_{i}\p_{j}\lap^{-1}(\n_{\ep}\u^{j}_{\ep})-\n_{\ep}\u_{\ep}^{j} \p_{j}\p_{i}\lap^{-1}(\n_{\ep}) \right]\\
&\quad + \int\left(\frac{1}{2}\int |\na Q_{\ep}|^{2}+\frac{1}{2}  tr(Q_{\ep}^{2})(1+\frac{c_{_{*}}}{2}tr(Q_{\ep}^{2}))\right)\left(\p_{i}\phi \p_{i}\lap^{-1}(\n_{\ep})+\phi \n_{\ep} \right)\\
&\quad -\int \p_{j}Q_{\ep} \p_{i}Q_{\ep} \left(\p_{j}\phi \p_{i}\lap^{-1}(\n_{\ep})+\phi\p_{j} \p_{i}\lap^{-1}(\n_{\ep})\right)\\
&\quad +\int \left(Q_{\ep}^{ik} \lap Q_{\ep}^{kj} -\lap Q_{\ep}^{ik} Q_{\ep}^{kj}
 \right)\left(\p_{j}\phi \p_{i}\lap^{-1}(\n_{\ep})+\phi \p_{j} \p_{i}\lap^{-1}(\n_{\ep}) \right)\\
 &\quad+ \int  \si_{*}c_{\ep}^{2}Q_{\ep}^{ij}\left(\p_{j}\phi \p_{i}\lap^{-1}(\n_{\ep})+\phi \p_{j} \p_{i}\lap^{-1}(\n_{\ep}) \right)\\
&\quad-\ep^{2}\int   \n_{\ep}\u^{i}_{\ep}\phi \p_{i}\lap^{-1}(\div(\textbf{1}_{\mathcal{O}} \na\n_{\ep}))- \na \n_{\ep}\cdot \na \u_{\ep}^{i}\phi\p_{i}\lap^{-1}(\n_{\ep})\\
&\quad+\ep \int  \n_{\ep}\u^{i}_{\ep}\phi \p_{i}\lap^{-1}(2\n_{\ep}-\n_{0})\\
&=: \sum_{n=1}^{11}T_{n}^{\ep},\ea\ee
where $T_{n}^{\ep}$ denotes the $n^{th}$ integral  on the right hand side of \eqref{4.3}.

On the other  hand,  if    we    multiply  $\eqref{n6f}_{2}$  by  $\phi\p_{i}\lap^{-1}(\n)$, we obtain
 \be\la{4.5}\ba   \int\phi \n \mathbb{F} &=-\int \p_{i}\lap^{-1}(\n) \p_{i}\phi \left(\de\overline{\n^{4}}+\overline{\n^{\g}} -(\mu+\lambda )\div \u\right)-\int \n g_{2}^{i}\phi\p_{i}\lap^{-1}(\n)\\
&\quad+\mu\int \left( \p_{j}\u^{i}\p_{i}\lap^{-1}(\n) \p_{j}\phi- \u^{i}\p_{j}\p_{i}\lap^{-1}(\n)\p_{j}\phi+  \n \u \cdot\na \phi\right)\\
&\quad-\int \n\u^{j}\u^{i}\p_{j}\phi \p_{i}\lap^{-1}(\n)+\int \u^{i}\phi \left[\n\p_{i}\p_{j}\lap^{-1}(\n\u^{j})-\n\u^{j} \p_{j}\p_{i}\lap^{-1}(\n) \right]\\
&\quad + \int\left(\frac{1}{2}\int |\na Q|^{2}+\frac{1}{2}  tr(Q^{2})(1+\frac{c_{_{*}}}{2}tr(Q^{2}))\right)\left(\p_{i}\phi \p_{i}\lap^{-1}(\n)+\phi \n \right)\\
&\quad -\int \p_{j}Q \p_{i}Q\left(\p_{j}\phi \p_{i}\lap^{-1}(\n)+\phi\p_{j} \p_{i}\lap^{-1}(\n)\right)\\
&\quad +\int \left(Q^{ik} \lap Q^{kj} -\lap Q^{ik} Q^{kj}
 \right)\left(\p_{j}\phi \p_{i}\lap^{-1}(\n)+\phi \p_{j} \p_{i}\lap^{-1}(\n) \right)\\
 &\quad+ \int  \si_{*}c^{2}Q^{ij}\left(\p_{j}\phi \p_{i}\lap^{-1}(\n)+\phi \p_{j} \p_{i}\lap^{-1}(\n) \right)\\
&=: \sum_{n=1}^{9}T_{n}. \ea\ee
In terms of \eqref{4.3} and \eqref{4.5},  to  prove \eqref{4.6}  it suffices to check
\bnn\lim_{\ep\rightarrow0} T_{n}^{\ep}=T_{n}\,\,(n=1, 2,  \cdots, 9)\quad {\rm and }\quad\lim_{\ep\rightarrow0} T_{n}^{\ep}=0\,\,(n=10,11).\enn
In fact,  by   the Mikhlin multiplier theory (cf. \cite{stein}), and the Rellich-Kondrachov compactness  theorem (cf. \cite{evans}), 
one has  \be\la{e2}  \p_{j}\p_{i}\lap^{-1}(h_{\ep})  \rightharpoonup\p_{j}\p_{i}\lap^{-1}(h)\quad {\rm in }\,\, L^{p},\quad  \p_{i}\lap^{-1}(h_{\ep})  \rightarrow \p_{i}\lap^{-1}(h)\quad {\rm in }\,\, L^{q}, \ee where $q<(1/p-1/3)^{-1}$ if $p<3$ and $q\le \infty$ if $p>3.$ 
By \eqref{e2}, as well as   \eqref{b12} and \eqref{b11},   we  have
\be\la{k4}\ba T^{\ep}_{8}
 &= \int \left(Q_{\ep}^{ik} \lap Q_{\ep}^{kj} -\lap Q_{\ep}^{ik} Q_{\ep}^{kj}
 \right)\left(\p_{j}\phi \p_{i}\lap^{-1}(\n_{\ep})+\phi \p_{j} \p_{i}\lap^{-1}(\n_{\ep}) \right)\\
 &=\int \left(\na Q_{\ep}^{ik} \na Q_{\ep}^{kj} -\na Q_{\ep}^{ik}\na  Q_{\ep}^{kj}
 \right) \left(\p_{j}\phi \p_{i}\lap^{-1}(\n_{\ep})+\phi \p_{j} \p_{i}\lap^{-1}(\n_{\ep}) \right)\\
 &\quad +\int \left(  Q_{\ep}^{ik} \na Q_{\ep}^{kj} -\na Q_{\ep}^{ik} Q_{\ep}^{kj}
 \right)\na\left(\p_{j}\phi \p_{i}\lap^{-1}(\n_{\ep})+\phi \p_{j} \p_{i}\lap^{-1}(\n_{\ep}) \right)\\
  &=\int \left(\na Q_{\ep}^{ik} \na Q_{\ep}^{kj} -\na Q_{\ep}^{ik}\na  Q_{\ep}^{kj}
 \right) \left(\p_{i}\phi \p_{i}\lap^{-1}(\n_{\ep})+\phi \n_{\ep} \right)\\
 &\rightarrow \int \left(\na Q^{ik} \na Q^{kj} -\na Q^{ik} \na Q^{kj}
 \right) \left(\p_{i}\phi \p_{i}\lap^{-1}(\n)+\phi \n \right)\\
 &=\int \left(Q  \lap Q  -\lap Q  Q
 \right)\left(\p_{i}\phi \p_{i}\lap^{-1}(\n)+\phi \n  \right)\\
 &=T_{8},\ea\ee  where  the third equality  is valid after summing up $i,j=1,2,3, $ due to the fact that the matrix  $Q$ is  symmetric and the following computation:
 \bnn \ba &\int \left(  Q_{\ep}^{ik} \na Q_{\ep}^{kj} -\na Q_{\ep}^{ik} Q_{\ep}^{kj}
 \right)\na \left(\p_{j}\phi \p_{i}\lap^{-1}(\n_{\ep})+\phi \p_{j} \p_{i}\lap^{-1}(\n_{\ep}) \right)\\
 & +\int \left(  Q_{\ep}^{jk} \na Q_{\ep}^{ki} -\na Q_{\ep}^{jk} Q_{\ep}^{ki}
 \right)\na \left(\p_{i}\phi \p_{j}\lap^{-1}(\n_{\ep})+\phi \p_{i} \p_{j}\lap^{-1}(\n_{\ep}) \right)\\
 =& \int \left(  Q_{\ep}^{ik} \na Q_{\ep}^{kj} -\na Q_{\ep}^{jk} Q_{\ep}^{ki}
 \right)\na \left(\p_{j}\phi \p_{i}\lap^{-1}(\n_{\ep})+\phi \p_{j} \p_{i}\lap^{-1}(\n_{\ep}) \right)\\
 & +\int \left(  Q_{\ep}^{jk} \na Q_{\ep}^{ki} -\na Q_{\ep}^{ik} Q_{\ep}^{kj}
 \right)\na \left(\p_{i}\phi \p_{j}\lap^{-1}(\n_{\ep})+\phi \p_{i} \p_{j}\lap^{-1}(\n_{\ep}) \right)\\
 &
 = 0.\ea\enn
Next,  utilizing \eqref{b6}, \eqref{b12}-\eqref{b14},   and \eqref{e2} again,     we deduce   that   
\bnn \lim_{\ep\rightarrow0} T^{\ep}_{n}=T_{n},\quad {\rm for}\,\,\,i=1,2,3,4,6,7,9,\enn
   and    \bnn \lim_{\ep\rightarrow0} T^{\ep}_{n}=T_{n},\quad {\rm for}\,\,\,i=10,11.\enn
In order to justify  \be\la{s15}\lim_{\ep\rightarrow0} T^{\ep}_{5}=T_{5}, \ee  we present the    following Lemma (cf. \cite{fei5}):
\begin{lemma}[div-curl]\la{lem4.4}  Let $\frac{1}{r_{1}}+\frac{1}{r_{2}}=\frac{1}{r}$ and $1\le r,r_{1},r_{2}<\infty.$
 Suppose that
\bnn v_{\ep}\rightharpoonup v\,\,{\rm in}\,\,L^{r_{1}}\quad {\rm and}\quad w_{\ep}\rightharpoonup w\,\,{\rm in}\,\,L^{r_{2}}.\enn
Then,  \bnn v_{\ep}\p_{i}\p_{j}\lap^{-1}(w_{\ep})-w_{\ep}\p_{i}\p_{j}\lap^{-1}(v_{\ep})\rightharpoonup v\p_{i}\p_{j}\lap^{-1}(w)-w\p_{i}\p_{j}\lap^{-1}(v)\quad{\rm in}\,\,L^{r},\,\,\,(i,j=1,2,3).\enn
\end{lemma}
Taking    $v_{\ep}=\n_{\ep}\u_{\ep}^{j}$ and $w_{\ep}=\n_{\ep}$,   we obtain  \eqref{s15} by Lemma \ref{lem4.4}.   The proof of Lemma \ref{lem4.3} is thus completed.

\bigskip

 \section{Vanishing Artificial Pressure}
 
In this   section, we will complete the proof of Theorem \ref{t} by  taking the limit as  $\de\rightarrow 0$ in the solutions
$(\n_{\de},c_{\de},\u_{\de},Q_{\de})$  obtained in Theorem \ref{t4.1}.

\subsection{Refined estimates on energy function}
We first derive  the refined   estimates on  $(\n_{\de},c_{\de},\u_{\de},Q_{\de})$   uniform in $\de$, which helps  us  relax the restriction on $\g$.
\begin{proposition}\la{p2} Let $(\n_{\de},c_{\de},\u_{\de},Q_{\de})$  be   the solution  obtained in Theorem \ref{t4.1}.       Then, under the assumptions   in Theorem \ref{t}, the following inequality holds for all $s\in (1,\frac{3}{2})$,
 \be\la{ss1}  \|\de\n_{\de}^{4}+\n_{\de}^{\g} \|_{L^{s}}+\|Q_{\de}  \|_{L^{6}}+\|\u_{\de}\|_{H^{1}_{0}}+\|\na Q_{\de}\|_{H^{1}}+\|c_{\de}\|_{L^{\infty}}\le C,\ee where, and in what follows,   the constant $C$ is independent of $\de.$
\end{proposition}
The proof of Proposition \ref{p2}  borrows some   ideas developed  in \cite{freh,lw,mpm,pw}. We present the details below through      several lemmas.

\begin{lemma}\la{lem5.2x} Let $(\n_{\de},c_{\de},\u_{\de},Q_{\de})$ be the solution  obtained in Theorem \ref{t4.1}. Then  there  are   constants $C$ and $C_{1}$   independent of $\de$ such that,
\be\la{1.3ff}\ba   \|c\|_{L^{\infty}}^{6}&+ \int \left(|\na \u_{\de}|^{2}  +  |\lap Q_{\de}|^{2}+   |\na Q_{\de}|^{2}
+    |Q_{\de}|^{6} \right) \\
&\le   C\left(1+\|\n_{\de}\u_{\de}\|_{L^{1}}+ m_{2}\|\n_{\de}\u_{\de}\|_{L^{1}}^{3C_{1}}\right),\ea\ee provided that   $m_{2}$ is sufficiently small.\end{lemma}
\begin{proof}
Using the same computation as   that in \eqref{pp1}-\eqref{pp3},   we multiply    $\eqref{n6}_{3}$ by $\u_{\de}$ and        $\eqref{n6}_{4}$ by $-\lap Q_{\de}+Q_{\de}+c_{*}Q_{\de}tr(Q_{\de}^{2})$ to  deduce
\bnn\ba &\mu \int |\na \u_{\de}|^{2}+(\lambda +\mu)\int |{\rm div} \u_{\de}|^{2}+\int  |\lap Q_{\de}|^{2}+  \int  |\na Q_{\de}|^{2}
 + c_{*} \int \left(|Q_{\de}|^{4}+c_{*}|Q_{\de}|^{6}\right) \\
&\le   \int  \n_{\de} g_{2}   \cdot \u_{\de}  +  \int b\left(Q_{\de}^{2}-\frac{1}{3} tr(Q_{\de}^{2})\mathbb{I}\right):\left(-\lap Q_{\de}+Q_{\de}+c_{*}Q_{\de}tr(Q_{\de}^{2})\right) \\
 &\quad  + \int  g_{3} :\left(-\lap Q_{\de}+Q_{\de}+c_{*}Q_{\de}tr(Q_{\de}^{2})\right)\\
 &\quad-  \si_{*}\int c_{\de}^{2}Q_{\de} :\na \u_{\de} +\int \frac{(c_{\de}-c_{*})}{2}Q_{\de}:(\lap Q_{\de}-Q_{\de}-c_{*}Q_{\de}tr(Q_{\de}^{2})).\ea\enn
From   \eqref{0}  it follows that
 \bnn\ba  & \left|\int b\left(Q_{\de}^{2}-\frac{1}{3}tr(Q_{\de}^{2})\mathbb{I}\right):\left(-\lap Q_{\de}+Q_{\de}+c_{*}Q_{\de}tr(Q_{\de}^{2})\right) \right|\\
 &\quad  + \left|\int  g_{3} :\left(-\lap Q_{\de}+Q_{\de}+c_{*}Q_{\de}tr(Q_{\de}^{2})\right)\right|\\
 &\le C (1+\|Q_{\de}\|_{L^{4}}^{2}+  \|g_{3}\|_{L^{2}})\left(\|\lap Q_{\de}\|_{L^{2}}+ \|Q_{\de}\|_{L^{6}}^{3}+1\right)\\
 &\le  C+\frac{ c_{*}^{2}}{4}
  \int  |Q_{\de}|^{6}
 + \frac{1}{4}\int  |\lap Q_{\de}|^{2},\ea\enn
 and \bnn\ba  & \left|-  \si_{*}\int c_{\de}^{2}Q_{\de} :\na \u_{\de} +\int \frac{(c_{\de}-c_{*})}{2}Q_{\de}:(\lap Q_{\de}-Q_{\de}-c_{*}Q_{\de}tr(Q_{\de}^{2}))\right| \\
 &\le   C\|c_{\de}\|_{L^{\infty}}^{2}\|Q_{\de}\|_{L^{2}}\|\na\u_{\de}\|_{L^{2}}+C (1+ \|c_{\de}\|_{L^{\infty}})
 \left(\|Q_{\de}\|_{L^{2}}\|\lap Q_{\de} \|_{L^{2}}+  \|Q_{\de}\|_{L^{6}}^{4}+1\right)\\
 &\le C  +C  \|c_{\de}\|_{L^{\infty}}^{6} + \frac{  c_{*}^{2}}{4}
  \int  |Q_{\de}|^{6} + \frac{\mu}{2}\int |\na \u_{\de}|^{2}
 + \frac{1}{4}\int  |\lap Q_{\de}|^{2}.\ea\enn
The last three inequalities provide us
\be\la{1.3h}\ba &  \int |\na \u_{\de}|^{2} +\int  |\lap Q_{\de}|^{2}+  \int  |\na Q_{\de}|^{2}
 +   \int  |Q_{\de}|^{4}  \le    C+C  \|c_{\de}\|_{L^{\infty}}^{6}+\|\n_{\de}\u_{\de}\|_{L^{1}}.\ea\ee
With the aid of  \eqref{1.3h} and   \eqref{3.3}, we choose  $m_{2}$ sufficiently small such that
\bnn\ba\|c_{\de}\|_{L^{\infty}}&\le C(1+\|\u_{\de}\|_{L^{6}})^{C_{1}}(1+\|g_{1}\|_{L^{\infty}})m_{2}\\
&\le C(1+\|c_{\de}\|_{L^{\infty}}^{3}+\|\n_{\de}\u_{\de}\|_{L^{1}}^{\frac{1}{2}})^{C_{1}} m_{2}\\
&\le 2C+Cm_{2}\|\n_{\de}\u_{\de}\|_{L^{1}}^{\frac{C_{1}}{2}},\ea\enn
which together with  \eqref{1.3h} lead  to   the desired estimate \eqref{1.3ff}. \end{proof}

\begin{lemma}\la{lem5.1} Let $(\n_{\de},c_{\de},\u_{\de},Q_{\de})$  be   the solution obtained in Theorem \ref{t4.1}. Then, for any $s\in (1,\frac{3}{2}),$  the following  inequality holds true
 \be\ba\la{q1a}  \|\de \n_{\de}^{4}+\n_{\de}^{\g}\|_{L^{s}} \le  C\left(1+\|\n_{\de} |\u_{\de}|^{2} \|_{L^{s}} + m_{2}\|\n_{\de}\u_{\de}\|_{L^{1}}^{3C_{1}}\right), \ea\ee  provided  that $m_{2}$ is sufficiently small. \end{lemma}
 \begin{proof} As in  Lemma \ref{lem2.2g}, we introduce the Bogovskii operator \bnn \mathcal{B}:=\mathcal{B}(h-(h)_{\mathcal{O}})\quad{\rm with}\quad(h)_{\mathcal{O}}=|\mathcal{O}|^{-1}\int h.\enn
 Then, for any $h\in L^{\frac{s}{s-1}}$ with $s\in (1,\frac{3}{2})$,  Lemma \ref{lem2.2g} implies
 \be\la{s34}  \|\mathcal{B} \|_{L^{\infty}}+\|\na\mathcal{B} \|_{L^{2}}
+\|\na\mathcal{B} \|_{L^{\frac{s}{s-1}}} \le C\|h\|_{L^{\frac{s}{s-1}}}.\ee
Multiplying   $\eqref{n6}_{3}$ by $\mathcal{B}$ gives
\be\la{q5.3}\ba  & \int\left(\de \n_{\de}^{4}+\n_{\de}^{\g}  \right) h\\
&=(h)_{\mathcal{O}}\int \left(\de \n_{\de}^{4}+\n_{\de}^{\g}\right) -\int \n g_{2} \cdot \mathcal{B}(h-(h)_{\mathcal{O}})+ \int \mathbb{S}_{ns}:\na \mathcal{B} \\
&\quad -\int \n_{\de} \u_{\de}\otimes \u:\na \mathcal{B} +\int \mathbb{S}_{1} :\na \mathcal{B} + \int \mathbb{S}_{2}:\na \mathcal{B} \\
&\le C\|h\|_{L^{\frac{s}{s-1}}} \|\de \n_{\de}^{4}+\n_{\de}^{\g}\|_{L^{1}}+C\|\mathcal{B} \|_{L^{\infty}}   +C\|\na \u_{\de}\|_{L^{2}}\|\na \mathcal{B} \|_{L^{2}}\\
 &\quad +C\left(\|\n_{\de} |\u_{\de}|^{2}+c^{2}Q_{\de}+|Q_{\de}|^{2} +|Q_{\de}|^{4}+|\na Q_{\de}|^{2}+Q_{\de}\lap Q_{\de}\|_{L^{s}} )\|\na\mathcal{B} \|_{L^{\frac{s}{s-1}}}\right)\\
 &\le C\|h\|_{L^{\frac{s}{s-1}}} \left(1+\|\de \n_{\de}^{4}+\n_{\de}^{\g}\|_{L^{1}}+\|\n_{\de} |\u_{\de}|^{2}\|_{L^{s}}+\|\na \u_{\de}\|_{L^{2}}\right)\\
 &\quad  +C\|h\|_{L^{\frac{s}{s-1}}} \left( \|c^{2}Q_{\de}+|Q_{\de}|^{2} +|Q_{\de}|^{4}+|\na Q_{\de}|^{2}+Q_{\de}\lap Q_{\de}\|_{L^{s}} \right),
\ea\ee
where, for  the last  inequalities,  we have used   \eqref{b0}, \eqref{b1}, \eqref{0}, \eqref{s34} and $\|\n_{\de}\|_{L^{1}}=m_{1}$.
Due to the  arbitrariness of $h\in L^{\frac{s}{s-1}}$, it yields from    \eqref{q5.3}  that
\be\la{q5.3a}\ba   \|\de \n_{\de}^{4}+\n_{\de}^{\g} \|_{L^{s}}
 &\le C  \left(1+\|\de \n_{\de}^{4}+\n_{\de}^{\g}\|_{L^{1}}+ \|\n_{\de} |\u_{\de}|^{2}\|_{L^{s}}+\|\na \u_{\de}\|_{L^{2}}\right)\\
 &\quad +C \|c^{2}Q_{\de}+|Q_{\de}|^{2} +|Q_{\de}|^{4}+|\na Q_{\de}|^{2}+Q_{\de}\lap Q_{\de}\|_{L^{s}} \\
  &\le  \frac{1}{2} \|\de \n_{\de}^{4}+\n_{\de}^{\g}\|_{L^{s}}+\left(1+  \|\n_{\de} |\u_{\de}|^{2}\|_{L^{s}}+\|\na \u_{\de}\|_{L^{2}}\right)\\
 &\quad+C \|c^{2}Q_{\de}+|Q_{\de}|^{2} +|Q_{\de}|^{4}+|\na Q_{\de}|^{2}+Q_{\de}\lap Q_{\de}\|_{L^{s}}.\ea\ee
 Since $s\in (1,\frac{3}{2}),$  one has \bnn\ba &\| c_{\de}^{2}Q_{\de}+|Q_{\de}|^{2} +|Q_{\de}|^{4}+|\na Q_{\de}|^{2}+Q_{\de}\lap Q_{\de}\|_{L^{s}}\\
&\le C \left(1+\|c_{\de}\|_{L^{\infty}}^{3}+ \|Q_{\de}\|_{L^{6}}^{4}+ \|\na Q_{\de}\|_{L^{2}}^{2}+\|\lap Q_{\de}\|_{L^{2}}^{2}\right)\\
&\le C \left(1+\|c_{\de}\|_{L^{\infty}}^{3}+ \|Q_{\de}\|_{L^{6}}^{6}+ \|\na Q_{\de}\|_{L^{2}}^{2}+\|\lap Q_{\de}\|_{L^{2}}^{2}\right).\ea\enn
Therefore, substituting it  into \eqref{q5.3a} and utilizing  \eqref{1.3ff}  we obtain  \eqref{q1a}. \end{proof}

Next,  we shall deduce a  weighted estimate  on both  the pressure and kinetic energy.

 \begin{lemma} \la{c1} Let $(\n_{\de},c_{\de},\u_{\de},Q_{\de})$  be   the solution  obtained in Theorem \ref{t4.1}. Then, for any $\a \in (0,1)$ and $s\in (1,\frac{3}{2})$,  the following inequality  holds true
 \be\la{r5-3}\ba  &\sup_{x^{*}\in \overline{\mathcal{O}}}\int  \frac{\left(\de \n_{\de}^{4}+\n_{\de}^{\g} +\n_{\de} |\u_{\de}|^{2}\right)(x)}{|x-x^{*}|^{\a}}dx\\&\le  C\left(1+\|\n_{\de} |\u_{\de}|^{2} \|_{L^{s}}+  m_{2}\|\n_{\de}\u_{\de}\|_{L^{1}}^{3C_{1}}\right),\ea\ee  provided that $m_{2}$ is sufficiently small.
\end{lemma}
\begin{proof} We  adopt  some ideas   in     \cite{freh,lw,mpm}, and  divide the process into two cases.
\medskip

\noindent
{\it Case 1: The boundary  point case  $x^{*}\in \p\mathcal{O}.$}

As in \cite[Exercise 1.15]{zie}  we introduce a  function $\phi(x)\in C^{2}(\overline{\mathcal{O}})$ that behaviors like  the distance  when $x\in \mathcal{O}$ is near   the boundary and  is extended   smoothly  to the whole domain $\mathcal{O},$   and moreover,    \be\la{x7}\left\{\ba& \phi(x)>0\,\,\,{\rm in}\,\,\mathcal{O}\,\,\,{\rm and}\,\,\,\phi(x)=0\,\,\,{\rm on}\,\,\,\p\mathcal{O},\\
&|\phi(x)|\ge k_{1} \,\,\,{\rm if}\,\,\,x\in \mathcal{O}\,\,\,{\rm and}\,\,\,{\rm dist}(x,\,\p\mathcal{O})\ge k_{2},\\
&  \na \phi=\frac{x-\tilde{x}}{\phi(x)} =\frac{x-\tilde{x}}{|x-\tilde{x}|} \,\,\,{\rm if}\,\,\,x\in \mathcal{O}\,\,\,{\rm and}\,\,\,{\rm dist}(x,\,\p\mathcal{O})=|x-\tilde{x}|\le k_{2}, \ea\right.\ee
where the  positive constants $k_{1}$  and $k_2$ are given.
Following   \cite{freh}, we define
\be\la{r10} \xi(x)=\frac{\phi(x)\na\phi(x)}{ \left(\phi(x)+|x-x^{*}|^{\frac{2}{2-\a}}\right)^{ \a}}
\quad {\rm with}\quad x,\,x^{*}\in \overline{\mathcal{O}}.\ee
It follows  from   \eqref{x7} that,
for all   points $x$ satisfying ${\rm dist}(x,\,\p\mathcal{O})\le k_{2}$,   \be\la{z009}  \phi<\phi+|x-x^{*}|^{\frac{2}{2-\a}} \le  C|x-x^{*}|,\ee owing to $\frac{2}{2-\a}>1$.
By \eqref{r10}, a  careful computation   gives
 \be\la{r15} \ba\p_{j}\xi^{i}
 &=\frac{\phi\p_{j}\p_{i}\phi}{\left(\phi+|x-x^{*}|^{\frac{2}{2-\a}}\right)^{ \a}}+\frac{  \p_{j}\phi\p_{i}\phi}{\left(\phi+|x-x^{*}|^{\frac{2}{2-\a}}\right)^{\a}}\\
 &\quad- \a\frac{\phi\p_{i}\phi\p_{j}\phi}{\left(\phi+|x-x^{*}|^{\frac{2}{2-\a}}\right)^{\a+1}}- \a\frac{\phi\p_{i}\phi\p_{j}|x-x^{*}|^{\frac{2}{2-\a}}}{\left(\phi+|x-x^{*}|^{\frac{2}{2-\a}}\right)^{\a+1}},\quad i,j=1,2,3.\ea\ee
 Thus,    $ |\na \xi|\in L^{q}$ for all $q\in [2,\frac{3}{\a})$.
In view of  \eqref{r10}-\eqref{r15},  one has
\be\la{x8}\ba  C + \frac{C}{|x-x^{*}|^{\a}}\ge
 \div \xi &\ge -C + \frac{(1-\a)}{2}\frac{|\na \phi|^{2}}{\left(\phi+|x-x^{*}|^{\frac{2}{2-\a}}\right)^{\a}} \\
&\ge -C + \frac{C}{|x-x^{*}|^{\a}}.\ea\ee
 In addition, by \eqref{x7},
\be\la{x7a} \p_{j}\p_{i}\phi=\frac{\p_{i}(x-\tilde{x})^{j}}{\phi}-\frac{\p_{j}\phi\p_{i}\phi}{\phi}.\ee
Hence,
if we multiply   $\eqref{n6}_{3}$ by    $\xi$, we find
 \be\la{r11}\ba&\int \left(\de \n_{\de}^{4}+\n_{\de}^{\g}\right)\div \xi+ \int  \n_{\de}   \u_{\de} \otimes \u_{\de}  :\na  \xi\\
 &=  \int \left(\mathbb{S}_{ns}(\na \u_{\de})+\mathbb{S}_{1}(Q_{\de})+\mathbb{S}_{2}(c_{\de},Q_{\de})\right):\na  \xi-\int \n_{\de} g_{2} \cdot \xi.\ea\ee
Making use of  \eqref{c0}, \eqref{b0}, \eqref{b1}, \eqref{0},   \eqref{x8}, and the fact  $\xi\in W_{0}^{1,3}$, one deduces
\be\la{r13}\ba  & \left|\int \left(\mathbb{S}_{ns}(\na \u_{\de})+\mathbb{S}_{1}(Q_{\de})+\mathbb{S}_{2}(c_{\de},Q_{\de})\right):\na  \xi-\int \n_{\de} g_{2} \cdot \xi\right|\\
& \le C(\a) \left(1+\|\na \u_{\de} \|_{L^{2}} +\| |\na Q_{\de}|^{2}+|Q_{\de}|^{4} +|Q_{\de}||\lap Q_{\de}|+c_{\de}^{2}Q_{\de}\|_{L^{\frac{3}{2}}} \right)\\
& \le C(\a)\left(1+\|\na \u_{\de} \|_{L^{2}} +\|\lap Q_{\de} \|_{L^{2}}^{2} +\|\na Q_{\de} \|_{L^{2}}^{2}+\| Q_{\de}  \|_{L^{6}}^{4}+ \|c_{\de} \|_{L^{\infty}}^{3}\right),\ea\ee
and    \be\la{r12}\ba
 \int \left(\de \n_{\de}^{4}+\n_{\de}^{\g}\right)\div \xi\ \ge
 -C \int\left(\de \n_{\de}^{4}+\n_{\de}^{\g}\right)+C \int_{\mathcal{O}\cap B_{k_{2}}(x^{*})} \frac{\left(\de \n_{\de}^{4}+\n_{\de}^{\g} \right)}{|x-x^{*}|^{\a}}.\ea\ee
By  \eqref{x7a},  
\bnn \ba&\int \frac{\phi \n_{\de}  \u_{\de} \otimes \u_{\de} \p_{j}\p_{i}\phi}{\left(\phi+|x-x^{*}|^{\frac{2}{2-\a}}\right)^{\a}} =\int \frac{\n_{\de}|\u_{\de} |^{2}}{\left(\phi+|x-x^{*}|^{\frac{2}{2-\a}}\right)^{ \a}}-\int \frac{\n_{\de}|\u_{\de} \cdot\na \phi|^{2}}{\left(\phi+|x-x^{*}|^{\frac{2}{2-\a}}\right)^{\a}},\ea\enn
which along  with  \eqref{x7}, \eqref{z009}, \eqref{x7a} and  the Schwarz inequality
imply
\be\la{x9}\ba &\int \n_{\de} \u_{\de}\otimes \u_{\de} :\na  \xi\\
&=\int \frac{  \n_{\de}  |\u_{\de}|^{2}}{\left(\phi+|x-x^{*}|^{\frac{2}{2-\a}}\right)^{ \a}} - \a \int \frac{\phi \n_{\de}(\u_{\de}\cdot \na\phi)^{2}}{\left(\phi+|x-x^{*}|^{\frac{2}{2-\a}}\right)^{\a+1}}\\
&\quad - \a\int \frac{\phi\n_{\de}(\u_{\de}\cdot \na |x-x^{*}|^{\frac{2}{2-\a}})(\u_{\de}\cdot\na \phi)}{\left(\phi+|x-x^{*}|^{\frac{2}{2-\a}}\right)^{\a+1}}\\
&\ge (1- \a)\int \frac{  \n_{\de} |\u_{\de}|^{2}}{\left(\phi+|x-x^{*}|^{\frac{2}{2-\a}}\right)^{\a}}
- \a\int \frac{\phi\n_{\de} (\u_{\de}\cdot \na |x-x^{*}|^{\frac{2}{2-\a}})(\u_{\de}\cdot\na \phi)}{\left(\phi+|x-x^{*}|^{\frac{2}{2-\a}}\right)^{\a+1}}\\
&\ge \frac{(1- \a)}{2}\int \frac{\n_{\de} |\u_{\de}|^{2}}{\left(\phi+|x-x^{*}|^{\frac{2}{2-\a}}\right)^{\a}} -C \int \frac{\phi^{2}\n_{\de}|\u_{\de}|^{2}  |x-x^{*}|^{\frac{2\a}{2-\a}}}{\left(\phi+|x-x^{*}|^{\frac{2}{2-\a}}\right)^{\a+2}}\\ &\ge C \int_{\mathcal{O}\cap B_{k_{2}}(x^{*})} \frac{\n_{\de}|\u_{\de}|^{2}}{ |x-x^{*}|^{\a}} -C\|\n_{\de}|\u_{\de}|^{2}\|_{L^{1}}.\ea\ee
Therefore, the inequalities   \eqref{r11}-\eqref{x9} yield that, for some  $C$    independent of $x^{*}$,
\be\la{r5-1}\ba
&\int_{\mathcal{O}\cap B_{k_{2}}(x^{*})} \frac{\left(\de \n_{\de}^{4}+\n_{\de}^{\g}+\n_{\de} |\u_{\de} |^{2}\right)(x)}{|x-x^{*}|^{\a}}dx\\
& \le  C\left( \|\de \n_{\de}^{4}+\n_{\de}^{\g}\|_{L^{1}} +  \|\n_{\de} |\u_{\de}|^{2}  \|_{L^{1}}   \right) \\
& \quad+C\left( \|\na \u_{\de} \|_{L^{2}} +\|\lap Q_{\de} \|_{L^{2}}^{2} +\|\na Q_{\de} \|_{L^{2}}^{2}+\| Q_{\de}  \|_{L^{6}}^{4}+ \|c_{\de} \|_{L^{\infty}}^{3}\right)\\
&\le C\left(1+  \|\n_{\de} |\u_{\de}|^{2}  \|_{L^{s}}+ m_{2}\|\n_{\de}\u_{\de}\|_{L^{1}}^{3C_{1}}\right),\ea\ee
where the last inequality  follows from \eqref{q1a}   and \eqref{1.3ff}.

\medskip

\noindent
{\it Case 2: The interior point case  $x^{*}\in  \mathcal{O}.$}

We set  ${\rm dist}(x^{*},\,\p\mathcal{O})=3r>0$. Define   the smooth cut-off function
 \be\la{cut} \chi(x)= 1\,\,{\rm if}\,\,x\in B_{r}(x^{*}),\quad\chi(x)=0\,\,{\rm if}\,\,
 x\notin B_{2r}(x^{*}),
 \quad |\na \chi(x)|\le 2r^{-1}.\ee
  Multiplying $\eqref{n6}_{3}$ by
 $\frac{x-x^{*}}{|x-x^{*}|^{\a}}\chi^{2}$ yields
  \be\la{r8}\ba&\int \left(\de \n_{\de}^{4}+\n_{\de}^{\g}\right)  \frac{3-\a}{|x-x^{*}|^{\a}}\chi^{2} + \int \n_{\de} \u_{\de}\otimes \u_{\de} :\na  \left(\frac{x-x^{*}}{|x-x^{*}|^{\a}}\chi^{2}\right)\\
&=\int \left(\mathbb{S}_{ns}(\na \u_{\de})+\mathbb{S}_{1}(Q_{\de})+\mathbb{S}_{2}(c_{\de},Q_{\de})\right):\na  \left(\frac{x-x^{*}}{|x-x^{*}|^{\a}}\chi^{2}\right)-\int  \n_{\de} g_{2} \cdot \frac{x-x^{*}}{|x-x^{*}|^{\a}}\chi^{2}\\
&\quad
 -2\int \left(\de \n_{\de}^{4}+\n_{\de}^{\g}\right)  \chi \frac{\na\chi\cdot(x-x^{*})}{|x-x^{*}|^{\a}}.\ea\ee
A simple calculation  shows
\be\la{q2}\ba &\p_{i} \left(\frac{x^{j}-(x^{*})^{j}}{|x-x^{*}|^{\a}}\chi^{2}\right)\\
&= \frac{\p_{i}(x^{j}-(x^{*})^{j})}{|x-x^{*}|^{\a}}\chi^{2} -\a \frac{(x^{j}-(x^{*})^{j})(x^{i}-(x^{*})^{i})}{|x-x^{*}|^{\a+2}}\chi^{2}+2\chi \frac{x^{j}-(x^{*})^{j}}{|x-x^{*}|^{\a}}\p_{i} \chi,
\ea\ee  
thus, for some  constant $C$   independent of $x^{*}$ and $r,$
 \bnn \ba &\int \n_{\de} \u_{\de}\otimes \u_{\de}:\na \left(\frac{x-x^{*}}{|x-x^{*}|^{\a}}\chi^{2}\right)\\
 &\ge(1-\a)\int \frac{\n_{\de} |\u_{\de}|^{2}}{|x-x^{*}|^{\a}}\chi^{2}+2\int  \frac{ \chi\n_{\de} (\u_{\de}\cdot\na \chi )(\u_{\de}\cdot(x-x^{*}))}{|x-x^{*}|^{\a}}\\
 &\ge\frac{1-\a}{2}\int \frac{\n_{\de} |\u_{\de}|^{2}}{|x-x^{*}|^{\a}}\chi^{2}-C \int_{B_{2r}(x^{*})\backslash B_{r}(x^{*})} \frac{\n_{\de}|\u_{\de}|^{2}}{|x-x^{*}|^{\a}},\ea\enn where we have used  \be\la{s30}|\na \chi||x-x^{*}|\le 4,\quad \forall \,\,\,x\in B_{2r}(x^{*})\backslash B_{r}(x^{*}).\ee
Observe from \eqref{cut} and  \eqref{q2} that   $\na  \left(\frac{x-x^{*}}{|x-x^{*}|^{\a}}\chi^{2}\right) \in L^{q}$ for all $q\in [1,\frac{3}{\a}).$  
By the similar argument to \eqref{r13}, one has,  for some constant  $C$   independent of $r$,
 \bnn \ba &\left| \int\left(\mathbb{S}_{ns}(\na \u_{\de})+\mathbb{S}_{1}(Q_{\de})+\mathbb{S}_{2}(c_{\de},Q_{\de})\right):\na  \left(\frac{x-x^{*}}{|x-x^{*}|^{\a}}\chi^{2}\right)\right|\\
 &\quad  +\left|\int  \n_{\de} g_{2} \cdot \frac{x-x^{*}}{|x-x^{*}|^{\a}}\chi^{2} \right|\\
 &\le C \left(1+\|\na \u_{\de} \|_{L^{2}} +\|\lap Q_{\de} \|_{L^{2}}^{2} +\|\na Q_{\de} \|_{L^{2}}^{2}+\| Q_{\de}  \|_{L^{6}}^{4}+ \|c_{\de} \|_{L^{\infty}}^{3}\right)\ea\enn
and
\bnn\ba&\left|-2\int \left(\de \n_{\de}^{4}+\n_{\de}^{\g}\right)  \chi \frac{\na\chi\cdot(x-x^{*})}{|x-x^{*}|^{\a}} \right| \le   C\int_{B_{2r}(x^{*})\backslash B_{r}(x^{*})}
 \frac{\left(\de \n_{\de}^{4}+\n_{\de}^{\g}\right)}{|x-x^{*}|^{\a}},\ea \enn due to \eqref{s30}.
Therefore, taking  the above  inequalities  into accounts, utilizing \eqref{1.3ff},  we deduce from   \eqref{r8} that
  \be\la{r9*}\ba &\int_{B_{r}(x^{*})} \frac{\left(\de \n_{\de}^{4}+\n_{\de}^{\g}+\n_{\de}|\u_{\de}|^{2}\right)(x)}{ |x-x^{*}|^{\a}}dx \\
& \le   C \left(1+\|\na \u_{\de} \|_{L^{2}} +\|\lap Q_{\de} \|_{L^{2}}^{2} +\|\na Q_{\de} \|_{L^{2}}^{2}+\| Q_{\de}  \|_{L^{6}}^{4}+ \|c_{\de} \|_{L^{\infty}}^{3}\right)\\
&\quad + C \int_{B_{2r}(x^{*})\backslash B_{r}(x^{*})} \frac{\left(\de \n_{\de}^{4}+\n_{\de}^{\g}+\n_{\de}|\u_{\de}|^{2}\right)(x)} {|x-x^{*}|^{\a}}dx\\
&\le   C\left(1+  \|\n_{\de} |\u_{\de}|^{2}  \|_{L^{s}}+ m_{2}\|\n_{\de}\u_{\de}\|_{L^{1}}^{3C_{1}}\right)\\
&\quad + C \int_{B_{2r}(x^{*})\backslash B_{r}(x^{*})} \frac{\left(\de \n_{\de}^{4}+\n_{\de}^{\g}+\n_{\de}|\u_{\de}|^{2}\right)(x)} {|x-x^{*}|^{\a}}dx.\ea\ee
It remains   to estimate  the last  term  appeared in \eqref{r9*}.     To this end, we adopt  the ideas in \cite{lw} and  discuss   two cases: $(1)$  $x^{*}$ is far away from the boundary;  $(2)$  $x^{*}$ is close to the boundary.

 (1) Assume   ${\rm dist}(x^{*},\,\p\mathcal{O})= 3r\ge \frac{k_{2}}{2}$  with    $k_{2}$ given in \eqref{x7}.  Then
 \be\la{jj8}\ba  &\int_{B_{2r}(x^{*})\backslash B_{r}(x^{*})} \frac{\left(\de \n_{\de}^{4}+\n_{\de}^{\g}+\n_{\de}|\u_{\de}|^{2}\right)(x)} {|x-x^{*}|^{\a}}dx\\
  &\le  k_{2}^{-\a}   \int_{B_{2r}(x^{*})\backslash B_{r}(x^{*})}  \left(\de \n_{\de}^{4}+\n_{\de}^{\g}+\n_{\de}|\u_{\de}|^{2}\right) \\
  &\le   C\left(1+  \|\n_{\de} |\u_{\de}|^{2}  \|_{L^{s}}+ m_{2}\|\n_{\de}\u_{\de}\|_{L^{1}}^{3C_{1}}\right),\ea\ee
where  the last inequality   follows from   \eqref{q1a}.

(2)  Assume that $x^{*}\in \mathcal{O}$ is close to the boundary.  By \eqref{x7}, we have  \bnn |x^{*}-\tilde{x}^{*}|=dist(x^{*},\,\p\mathcal{O})= 3r<\frac{k_{2}}{2} \quad {\rm with}\,\,\,\,\tilde{x}^{*}\in \p\mathcal{O},\enn and hence,
       \be\la{bb1} 4 |x-x^{*}|\ge |x-\tilde{x}^{*}|,\quad \forall \,\,\, x\notin B_{r}(x^{*}).\ee
 Making use  of \eqref{bb1} and \eqref{r5-1},   we  get
 \be\la{jj9}\ba & C \int_{B_{2r}(x^{*})\backslash B_{r}(x^{*})} \frac{\left(\de \n_{\de}^{4}+\n_{\de}^{\g}+\n_{\de}|\u_{\de}|^{2}\right)(x)} {|x-x^{*}|^{\a}}dx\\
  & \le   C \int_{\mathcal{O}\cap B_{k_{2}}(\tilde{x}^{*})}\frac{\left(\de \n_{\de}^{4}+\n_{\de}^{\g}+\n_{\de}|\u_{\de}|^{2}\right)(x)} {|x-\tilde{x}^{*}|^{\a}}dx\\
 &\le  C\left(1+  \|\n_{\de} |\u_{\de}|^{2}  \|_{L^{s}}+ m_{2}\|\n_{\de}\u_{\de}\|_{L^{1}}^{3C_{1}}\right).\ea\ee

 In summary,
substituting   \eqref{jj8} and \eqref{jj9} back into  \eqref{r9*}  yields
 \be\la{r5-2}\ba  &\int_{B_{r}(x^{*})} \frac{\left(\de \n_{\de}^{4}+\n_{\de}^{\g} +\n_{\de} |\u_{\de}|^{2}\right)(x)}{|x-x^{*}|^{\a}}dx\\
 &\le  C\left(1+  \|\n_{\de} |\u_{\de}|^{2}  \|_{L^{s}}+ m_{2}\|\n_{\de}\u_{\de}\|_{L^{1}}^{3C_{1}}\right),\ea\ee where the constant $C$ is   independent of $x^{*}$.
The combination of \eqref{r5-1} with  \eqref{r5-2}  yoelds the desired estimate \eqref{r5-3}.
    \end{proof}

\begin{lemma} \la{lem2.3x} 
Assume that $\u\in H_{0}^{1}(\mathcal{O},\r)$ and $f(x)\ge0$ a.e. in $\mathcal{O}.$
Then there is a constant $C$   depending only on $|\mathcal{O}|$  such that
\be\la{y18} \int_{\mathcal{O}}|\u|^{2} fdx\le C \|\na \u\|_{H_{0}^{1}(\mathcal{O})}^{2}\int_{\mathcal{O}}\frac{f(x)}{|x-x^{*}|}dx,\ee
as long as the   right-hand side quantity is finite.
\end{lemma}
\begin{proof} The   proof is based on     the Green representation and integration by parts; see, e.g.,  \cite[Lemma 4]{pw}. \end{proof}
\begin{lemma}\la{lem5.2k} Let $\th  =\frac{\g-1}{2\g}\in (0,\, \frac{1}{2}).$ Then,
 \be\ba\la{q1ax}\int \n_{\de} |\u_{\de}|^{2(2-\th)}\le  C.\ea\ee\end{lemma}
\begin{proof}
Denote by
\bnn  \mathbb{A}=\int \n_{\de}  |\u_{\de}|^{2(2-\th)}.\enn
Noting  $\|\n_{\de} \|_{L^{1}}=m_{1},$ we have
\be\la{y8} \|\n_{\de} \u_{\de} \|_{L^{1}}\le \|\n_{\de}  |\u_{\de}|^{2(2-\th)}\|_{L^{1}}^{\frac{1}{2(2-\th)}}\|\n_{\de} \|_{L^{1}}^{\frac{3-2\th}{2(2-\th)}}\le C \mathbb{A}^{\frac{1}{2(2-\th)}}, \ee
 and
  \be\la{y9} \|\n_{\de} |\u_{\de} |^{2}\|_{L^{\frac{3}{2}}}\le
   \|\n_{\de}   |\u_{\de}|^{2(2-\th)}\|_{L^{1}}^{\frac{1}{2-\th}}
   \|\n_{\de}\|_{L^{1}}^{\frac{(1-\th)}{2-\th}}\le C\mathbb{A}^{\frac{1}{2-\th}}.\ee
Thanks  to  \eqref{y8}, it follows from \eqref{1.3ff} that
\be\la{r1s} \ba &\|c_{\de}\|_{L^{\infty}}^{6}+ \int \left(|\na \u_{\de}|^{2}  +  |\lap Q_{\de}|^{2}+   |\na Q_{\de}|^{2}
+    |Q_{\de}|^{6} \right)\\
 &\le  C\left(1+\mathbb{A}^{\frac{1}{2(2-\th)}}+ m_{2}\mathbb{A} ^{\frac{3C_{1}}{2(2-\th)}}\right).\ea\ee
A direct  calculation shows
\be\la{s40}  \frac{\n_{\de}  |\u_{\de}|^{2(1-\th)}}{|x-x^{*}|}
 =\left(\frac{\n_{\de} |\u_{\de}|^{2}}{|x-x^{*}|^{\a}}\right)^{1-\th} \left(\frac{\n_{\de}^{\g}}{|x-x^{*}|^{\a}}\right)^{\frac{\th}{\g}}
 \left(\frac{1}{|x-x^{*}|^{\a+\frac{\g(1-\a)}{(\g-1)\th}}}\right)^{\frac{(\g-1)\th}{\g}}.\ee
Noting that $\th=\frac{\g-1}{2\g},$ we  have   $\a+\frac{\g(1-\a)}{(\g-1)\th}\in (0,3)$ if   $\a\in (\frac{2\g-1}{\g^{2}},1)$, thus
 \bnn \int \frac{1}{|x-x^{*}|^{\a+\frac{\g(1-\a)}{(\g-1)\th}}} dx\le C.\enn
Therefore,  by  \eqref{r5-3}, \eqref{y8}, \eqref{y9}, we integrate \eqref{s40} to obtain
\be\la{y20}\ba \int \frac{\n_{\de}  |\u_{\de}|^{2(1-\th)}(x)}{|x-x^{*}|}dx
 &\le C+\int\frac{\n_{\de} |\u_{\de}|^{2}(x)}{|x-x^{*}|^{\a}}dx+\int\frac{\n_{\de}^{\g}(x)}{|x-x^{*}|^{\a}}dx\\
&\le C+C\int  \frac{\left(\de \n_{\de}^{4} +\n_{\de}^{\gamma} +\n_{\de}|\u_{\de}|^{2}\right)(x)}{|x-x^{*}|^{\a}}dx\\
&\le C\left(1+ \mathbb{A}^{\frac{1}{(2-\th)}}+  m_{2} \mathbb{A} ^{\frac{3C_{1}}{2(2-\th)}}\right).\ea\ee
From the definition of $ \mathbb{A}$,  \eqref{r1s} and  Lemma \ref{lem2.3x}, we obtain
\bnn\ba  \mathbb{A}&\le \|\na u\|_{L^{2}}^{2}\sup_{x^{*}\in \overline{\mathcal{O}}}
\int \frac{\n_{\de}  |\u_{\de}|^{2(1-\th)}(x)}{|x-x^{*}|}dx\\&
\le  C \left(1+\mathbb{A}^{\frac{1}{2(2-\th)}}+ m_{2}\mathbb{A} ^{\frac{3C_{1}}{2(2-\th)}}\right) \sup_{x^{*}\in \overline{\mathcal{O}}}
\int \frac{\n_{\de}  |\u_{\de}|^{2(1-\th)}(x)}{|x-x^{*}|}dx,\ea\enn
which together with  \eqref{y20} implies
 \bnn \ba\mathbb{A}
 \le 1+C\mathbb{A}^{\frac{3}{2(2-\th)}}+m_{2}\mathbb{A} ^{\frac{3C_{1}}{(2-\th)}}.\ea\enn
Since $\th\in (0,\frac{1}{2})$, we choose $m_{2}\le 1$ sufficiently small  to   conclude \eqref{q1ax}. The proof of Lemma \ref{lem5.2} is completed.
\end{proof}



Finally,  Proposition \ref{p2}   is a direct consequence of  Lemmas \ref{lem5.2x}-\ref{lem5.2k}.

\subsection{Vanishing artificial pressure}

Now we take the limit as $\delta\to 0$ in the spirit of \cite{novo, novo1, chen}.
Thanks to  \eqref{ss1},  the following estimate follows similarly to  \eqref{a26a}:
\be\la{a26as} \|\na c_{\de}\|_{L^{2}} +\|\lap c_{\de}\|_{L^{2}}  \le C.\ee
With \eqref{ss1} and \eqref{a26as} in hand,  we are allowed to  take the  following limits as $\de\to 0$,  subject  to a subsequence,
\be\la{6b10} (\na \u_{\de},\,\na^{2} c_{\de},\,\na^{2} Q_{\de})\rightharpoonup  (\na u,\,\na^{2} c,\,\,\na^{2}Q)\,\,{\rm in}\,\, L^{2},  \ee
\be\la{6b11}  \u_{\de}  \rightarrow   \u\,\,\,{\rm in}\quad L^{p_{1}},  \quad (c_{\de},\, Q_{\de}) \rightarrow    (c,\,Q)\,\,\,\,{\rm in}\,\,\,\, W^{1,p_{1}}\,\,\,\,(1\le p_{1}<6),\ee
\be\la{6b15}   \de \n_{\de}^{4} \rightarrow 0\,\,\,\, {\rm in}\,\,\,\,\mathcal{D}',  \quad \n_{\de}  \rightharpoonup \n \,\,{\rm in}\,\, L^{\g s},\quad {\rm for\,\,\,all}\,\,\,s\in (1,\frac{3}{2}).\ee
As $\g>1$, we can choose $s\in (1,\frac{3}{2})$  such that $\g s>\frac{3}{2}$. Then,   from   \eqref{6b11}-\eqref{6b15} one has
\be\la{6b14a}  \n_{\de}  \u_{\de} \rightharpoonup  \n \u,\quad {\rm in}\,\,\,  L^{p_{2}}\quad{\rm for\,\,\,some} \quad p_{2}>\frac{6}{5},\ee
and from   \eqref{6b11}-\eqref{6b14a},  for some $p_{3}>1$,
\be\la{6b14}  \left\{\ba&\n_{\de}  \u_{\de}^{i}\u_{\de}^{j} \rightharpoonup  \n \u^{i}\u^{j},\\
&Q_{\de}^{ik}  \lap Q_{\de}^{kj} \rightharpoonup Q^{ik}  \lap Q^{kj},\\
 &Q_{\de}^{ik}  (\p_{k}\u_{\de}^{j}-\p_{j}\u_{\de}^{k}) \rightharpoonup Q^{ik}   (\p_{k}\u^{j}-\p_{j}\u^{k}),\ea\right. 
 \ee
 in  $L^{p_{3}}$. 
Using \eqref{6b10}-\eqref{6b14},  we  take $\de$-limit  in \eqref{n6}  and obtain the  equations in the sense of distributions:
 \begin{equation}\label{n7b}
\left\{\ba &\div (\n \u)=0,\\
& \u\cdot \na c-\lap c=g_{1}, \\
&\div(\n \u\otimes \u)+\na   \overline{\n^{\g}}
 -\div\left(\mathbb{S}_{ns}(\na\u)+ \mathbb{S}_{1}(Q)+\mathbb{S}_{2}(c,Q) \right)=\n  g_{2},\\
& \u\cdot \na Q+Q\o-\o Q +c_{*}Qtr(Q^{2})+\frac{(c-c_{*})}{2}Q-b\left(Q^{2}-\frac{1}{3}tr(Q^{2})
 \mathbb{I}\right)-\lap Q=g_{3},\\
 &\u=0,\,\,\,\frac{\p c}{\p n}=0,\,\,\,\frac{\p Q}{\p n}=0,\quad {\rm on} \,\,\,\p\mathcal{O}.\ea \right.
\end{equation}
Additionally, \eqref{0r} follows from
   \eqref{a19}-\eqref{a19a}, \eqref{6b11},  \eqref{6b15}  and \eqref{s32} below.

 \bigskip

Next,  we
define   an increasing and concave  function $T_{k}(z)\in C_{0}^{1}([0,\infty))$,  satisfying  \be\la{7.0}  T_{k}(z)=z\,\,\,{\rm if}\,\,  z\le k,\quad
T_{k}(z)=k+1\,\,\,{\rm if}\,\,  z\ge k+1.\ee
Clearly, for any $1\le p\le \infty,$
 \be\la{9.1} T_{k}(\n_{\de})\rightharpoonup \overline{T_{k}(\n)}\quad {\rm in}\quad L^{p}.\ee

 \begin{lemma}\la{lem5.4} Let  $(\n_{\de},\u_{\de},c_{\de},Q_{\de})$ be the solution   obtained   in Theorem \ref{t4.1}.  Then,
\be\la{7.1} \ba
& \lim_{\de\rightarrow0}\int T_{k}(\n_{\de})\left( \n_{\de}^{\g} -(2\mu+\lambda)  \div \u_{\de}\right) =
\int \overline{T_{k}(\n)}\left( \overline{\n^{\g}} -(2\mu+\lambda)\div \u \right),\ea\ee
where $T_{k}$ is defined in \eqref{7.0}.
 \end{lemma}

\begin{proof} With the help of \eqref{6b10}-\eqref{6b14},  we   may slightly modify the argument   of  Lemma \ref{lem4.3} to complete the proof of Lemma \ref{lem5.4}.  The detail is omitted  here.
\end{proof}

Since $T_{k}$ is concave, one has
\bnn  (\n_{\de}^{\g}-\n^{\g})\left(T_{k}(\n_{\de})-T_{k}(\n)\right)\ge \left(T_{k}(\n_{\de})-T_{k}(\n)\right)^{\g+1}.\enn
Then,  from \eqref{7.1}, \eqref{6b15},    $\overline{\n^{\g}}\ge \n^{\g}$, and  $T_{k}(\n)\le \overline{T_{k}(\n)},$ we obtain
 \be\ba\la{7.12} &(2\mu+\lambda) \lim_{\de\rightarrow0}\int \left(T_{k}(\n_{\de})\div \u_{\de}-\overline{T_{k}(\n)}\div \u\right)\\
  &=\lim_{\de\rightarrow0}\int \left( T_{k}(\n_{\de})\n_{\de}^{\g}  -\overline{T_{k}(\n)}\,\,\overline{\n^{\g}} \right)\\
  &= \lim_{\de\rightarrow 0}\int  (\n_{\de}^{\g} -\n^{\g}  )\left(T_{k}(\n_{\de})-T_{k}(\n)\right)+ \int (\overline{\n^{\g}}-\n^{\g})\left(T_{k}(\n)-\overline{T_{k}(\n)}\right) \\
  &\ge  \lim_{\de\rightarrow 0}\int (\n_{\de}^{\g} -\n^{\g})\left(T_{k}(\n_{\de})-T_{k}(\n)\right)\\
  &\ge \lim_{\de\rightarrow 0}\int \left(T_{k}(\n_{\de})-T_{k}(\n)\right)^{\g+1}.\ea\ee
Noticing that  $\div\u_{\de}\in L^{2}$ is  bounded uniformly in $\de$, and
\bnn \lim_{\de\rightarrow0}  \int  \left(T_{k}(\n_{\de})\right)^{2}\ge \int  \left(\overline{T_{k}(\n)} \right)^{2},\enn
one has
\be\ba\la{8.4} &2C\lim_{\de\rightarrow0} \| T_{k}(\n_{\de}) - T_{k}(\n)  \|_{L^{2}}\\
&\ge C\lim_{\de\rightarrow0}\left( \| T_{k}(\n_{\de}) - T_{k}(\n)  \|_{L^{2}}+\| T_{k}(\n) -\overline{T_{k}(\n)} \|_{L^{2}}\right)\\
 &\ge (2\mu+\lambda)\lim_{\de\rightarrow0} \int \left(T_{k}(\n_{\de})-T_{k}(\n) + T_{k}(\n) -\overline{T_{k}(\n)}\right)\div \u_{\de}\\
&=(2\mu+\lambda)\lim_{\de\rightarrow0}\int \left(T_{k}(\n_{\de})- \overline{T_{k}(\n)}\right)\div \u_{\de}\\
&=(2\mu+\lambda)\lim_{\de\rightarrow0}\int \left(T_{k}(\n_{\de})\div \u_{\de}-\overline{T_{k}(\n)}\div \u\right).\ea\ee
In terms of  \eqref{7.12} and \eqref{8.4},   it holds that
\be\la{s33} \lim_{\de\rightarrow 0}\|T_{k}(\n_{\de})-T_{k}(\n)\|_{L^{\g+1}} \le C,\ee where the constant $C$   is independent of $k$ and $\de.$

We remark that  \eqref{s33} measures     oscillation of the density, which helps us   prove   that $\eqref{n7b}_{1}$     holds in the sense of renormalized solutions as in \cite{novo}. 

\begin{lemma}{\rm \cite{novo}}\la{lem5.2} 
For the solution $(\n,\u)$, 
\be\la{wq1} {\rm div}(b(\n) \u) +(b'(\n)\n-b(\n))\div \u=0\quad {\rm in}\quad \mathcal{D}'(\mathbb{R}^{3}),\ee
where $b(z)=z$, or  $b\in C^{1}([0,\infty))$ with $b'(z)=0$ for large $z$.
  \end{lemma}
\begin{proof} 
Thanks to Lemma \ref{lem4.2},  we see that $(\n_{\de},\u_{\de})$ is  a renormalized solution.
If we multiply  the equation \eqref{wq}  satisfied by $(\n_{\de},\u_{\de})$  by  $T_{k}(\n_{\de})$, and use  \eqref{7.0}, \eqref{6b10}-\eqref{6b14},  \eqref{s33}, we conclude \eqref{wq1} by   taking $\de\rightarrow0$ and then $k\rightarrow \infty$. The detailed proof may be found in \cite{novo}. 
\end{proof}

 In order  to complete the proof of   Theorem \ref{t} we only need to  verify
\be\la{s32}  \overline{\n^{\g}}=\n^{\g}.\ee
To this end,  
it suffices to prove the strong  convergence of $\n_{\de}$ in $L^{1}$ by  \eqref{6b15}.
The idea is to compare the limit of the renormalized solution  $(\n_{\de},\u_{\de})$ with $(\n,\u).$ 
In more detail,  we    introduce
 \bnn L_{k}=\left\{\ba &z\ln z,\quad\quad\quad\quad z\le k;\\
 &z\ln k+z\int_{k}^{z}\frac{T_{k}(s)}{s^{2}}ds,\quad z\ge k.\\
 \ea\right.\enn
A direct computation shows that \bnn C([0,\infty))\cap C^{1}((0,\infty))\ni b_{k}(z)=L_{k}(z)-\left(\ln k+\int_{k}^{k+1}\frac{T_{k}(s)}{s^{2}}+1\right)z,\enn
and moreover,    $b'(z)=0$ if $z\ge k+1$ and  $b_{k}'(z)z-b_{k}(z)=T_{k}(z).$
In view of  Lemma \ref{lem4.2} and Lemma \ref{lem5.2}, one has
\bnn \ba 0&=\div (b_{k}(\n) \u)+ T_{k}(\n)\div \u\\
&=\div (L_{k}(\n) \u)+ T_{k}(\n)\div \u,\quad {\rm in}\,\,\,\mathcal{D}(\r),\ea\enn
and
\bnn \ba 0 =\div (L_{k}(\n_{\de}) \u_{\de})+ T_{k}(\n_{\de})\div \u_{\de},\quad {\rm in}\,\,\,\mathcal{D}(\r).\ea\enn Integration of  the difference of   above two equations  leads to
 \bnn \int \left(T_{k}(\n)\div \u-T_{k}(\n_{\de})\div \u_{\de}\right)=0, \enn
which along with \eqref{7.12} and the fact $\div \u \in L^{2}$ implies   \be\ba\la{8.4x} C\| T_{k}(\n) -\overline{T_{k}(\n)} \|_{L^{2}}
&\ge (2\mu+\lambda) \int \left(T_{k}(\n) -\overline{T_{k}(\n)}\right)\div \u\\
&=(2\mu+\lambda)  \int \left(T_{k}(\n_{\de})\div \u_{\de}-\overline{T_{k}(\n)}\div \u\right)\\
  &\ge  \int \left(T_{k}(\n_{\de})-T_{k}(\n)\right)^{\g+1}.\ea\ee
Recalling Proposition  \ref{p2}, we have  $\|\n_{\de}\|_{L^{\g}}  \le C.$  Thus, the limit
\be\la{ddf}\ba \lim_{k\rightarrow\infty} \| T_{k}(\n_{\de})-\n_{\de}\|_{L^{1}}&
 = \| T_{k}(\n_{\de})-\n_{\de}\|_{L^{1}(\{\n_{\de}\ge k\})} \\
 &\le 2\lim_{k\rightarrow\infty} \|\n_{\de}\|_{L^{\g}(\{\n_{\de}\ge k\})} \\&\le C\lim_{k\rightarrow\infty}k^{1-\g}\|\n_{\de}\|_{L^{\g}(\{\n_{\de}\ge k\})}^{\g} =0 \ea\ee   is   uniform  in $\de$.
 In a similar way,
 \be\la{7.6} \ba \lim_{k\rightarrow\infty}\| T_{k}(\n)-\n\|_{L^{1}}=0.\ea\ee
Making use of  \eqref{bb10}-\eqref{7.6}, and
 \be\la{bb10}\ba \|T_{k}(\n) - T_{k}(\n_{\de})\|_{L^{1}}
    &\le C \|T_{k}(\n)-\n+\n_{\de}-\overline{T_{k}(\n)}\|_{L^{1}} \\
 &\le C  \left(\|T_{k}(\n)-\n\|_{L^{1}}
 +\lim_{\de\rightarrow0}\|T_{k}(\n_{\de})-\n_{\de}\|_{L^{1}} \right),\ea\ee we conclude
\bnn \ba& \lim_{\de\rightarrow0}\|\n_{\de}-\n\|_{L^{1}}\\&
\le \lim_{k\rightarrow\infty}\lim_{\de\rightarrow0}\left(\|\n_{\de}-T_{k}(\n_{\de})\|_{L^{1}}
+\|T_{k}(\n_{\de})-T_{k}(\n)\|_{L^{1}}+\|T_{k}(\n)-\n\|_{L^{1}}\right)\\&
=0.\ea\enn
The proof of Theorem \ref{t} is completed.

%
%
%
%
%
 
 \bigskip

\section*{Acknowledgements}
The research of Z. Liang was supported by the fundamental research funds for central universities (JBK 2202045).
A.  Majumdar (AM) acknowledges support from the University of Strathclyde New Professors Fund and a University of Strathclyde Global Engagement Grant. AM is also supported by a Leverhulme International Academic Fellowship. AM gratefully acknowledges hospitality from the University of Oxford and IIT Bombay, through her OCIAM Visiting Fellowship and Visiting Professorship at IIT Bombay.
The research  of D. Wang was partially supported by the National Science Foundation under grant  DMS-1907519.


\bigskip

 \begin{thebibliography}{99}
\bibitem{ad} R. Adams,    Sobolev spaces, New York: Academic Press, 1975.

\bibitem{live2} A. Ahmadi; M. C. Marchetti; T. B. Liverpool, Hydrodynamics of isotropic and liquid crystalline active polymer solutions,
Phys. Rev. E 74 (2006), 061913.
 
 \bibitem{an}  L.  Antanovskii,  A phase field model of capillarity,  Phys. Fluids A   {7} (1995), 747-753.

\bibitem{AM} D. Anderson;  G.  McFadden;  A.  Wheeler, { Diffuse-interface methods in fluid mechanics. Annual review of fluid mechanics,}     Annu. Rev. Fluid Mech. Annual Reviews  {30}  (1998), 139-165.

\bibitem{B-M-2010} J. M. Ball;  A. Majumdar,
Nematic liquid crystals: from Maier-Saupe to a continuum theory,
Molecular Crystals and Liquid Crystals, 525 (2010), 1-11.

\bibitem{bat} G.  Batchelor, { An Introduction to Fluid Dynamics,}  Cambridge University Press, Cambridge, 1999.

\bibitem{B-E-1994} A. N. Beris;  B. J. Edwards,
Thermodynamics of Flowing Systems with Internal Microstructure,
Oxford University Press: New York, 1994.


\bibitem{B-T-Y-2014} M. L. Blow; S. P. Thampi;  J. M. Yeomans,
{Biphasic, lyotropic, active nematics}, Phys. Rev. Lett. 113 (2014), 248303.

    \bibitem{chen0} G. Chen; A. Majumdar; D. Wang; R. Zhang,   Global existence and regularity of solutions for the active liquid   crystal, J. Diff. Equa., 50(4), (2017), 202-239.

\bibitem{chen} G. Chen; A. Majumdar; D. Wang; R. Zhang,   Global weak solutions for the compressible active liquid crystal system, SIAM J. Math. Anal., 263, (2018), 3632-3675.

    \bibitem{liang} R.-M. Chen; Z Liang; D. Wang; R. Xu, Energy equality in compressible fluids with physical boundaries, SIAM J. Math. Anal., 52 (2020),  1363-1385.


\bibitem{ds} H.  Davis; L.  Scriven,   Stress and structure in fluid interfaces,  Adv. Chem. Phys.
 {49}  (1982), 357-454.

\bibitem{dre} W. Dreyer; J. Giesselmann;  C. Kraus,
  A compressible mixture model with phase transition,  Phys.  D.  {273-274}   (2014), 1-13.


\bibitem{eb}
D.  Edwards;  H. Brenner;  D. Wasan,  { Interfacial Transport Process and Rheology,} Butterworths/Heinemann, London, 1991.



\bibitem{evans}
L.   Evans,  Partial differential equations. Second edition, Graduate Studies in Mathematics   {19}. American Mathematical Society, Providence, RI, 2010.


 \bibitem{fei} E. Feireisl, {Dynamics of viscous compressible fluids.}  Oxford University Press (2004).

\bibitem{fei1} E. Feireisl; A. Novotn\'{y}; H. Petzeltov\'{a}, On the existence of globally defined weak solutions to the Navier-Stokes equations, J. Math. Fluid Mech. \textbf{3}  (2001),  358-392.

    \bibitem{fei5} E. Feireisl, A. Novotn\'{y},  Singular Limits in Thermodynamics of Viscous Fluids. Advances
in Mathematical Fluid Mechanics, Birkh\"{a}user, Basel, 2009.



 \bibitem{freh} J. Frehse;  M. Steinhauer; W. Weigant,  {The Dirichlet problem for steady viscous compressible flow in three dimensions}, J. Math. Pures Appl.  {97}  (2012),   85-97.

 \bibitem{genn} P. De Gennes; J. Prost, The physics of Liquid crystals, The Clarendon Press, Oxford University Press, New York, 1995.



\bibitem{gt} D. Gilbarg;  N.  Trudinger, {Elliptic Partial Differential Equations of Second Order,}  2nd edition, Grundlehren Math. Wiss., vol. 224, Springer-Verlag, Berlin, Heidelberg, New York, 1983.

\bibitem{gi} L. Giomi; L. Mahadevan; B. Chakraborty;  M.   Hagan, Banding, excitability and chaos in active
nematic suspensions, Nonlinearity, 25 (2012), 2245.

\bibitem{HMP2017}
D. Henao;  A.  Majumdar;  A.  Pisante,  
Uniaxial versus biaxial character of nematic equilibria in three dimensions,  
Calc. Var. Partial Differential Equations 56 (2017), no. 2, Paper No. 55.

\bibitem{jiang2} S. Jiang;  P. Zhang, { Global spherically symmetry solutions of the compressible isentropic Navier-Stokes equations},  Comm. Math. Phys.  215   (2001),  85-97.

\bibitem{jz} S. Jiang; C. Zhou, On the existence theory to the three-dimensional steady compressible Navier-Stokes equations, Ann. Inst. H. Poincare\'{e} Anal. Non Lin\'{e}aire 28 (2011) 485-498.

\bibitem{K-F-K-L-2008} J. Kierfeld;  K. Frentzel; P. Kraikivski;   R. Lipowsky,
{Active dynamics of filaments in motility assays}, Eur. Phys. J. Special Topics, 157 (2008), 123.

\bibitem{lian} W. Lian; R. Zhang, Global weak solutions to the active hydrodynamics of liquid crystals., J. Differential Equations 268(8) (2020),   4194-4221.

\bibitem{lw} Z. Liang; D.  Wang,   Stationary Cahn-hilliard-Navier-Stokes equations for the diffuse interface model of compressible flows,  Math. Mod. Meth. Appl. Sci.  {30} (2020), 2445-2486.

\bibitem{lw1} Z. Liang; D.  Wang,   Stationary weak solutions for compressible Cahn-Hilliard/Navier-Stokes equations, 
 J. Nonlinear Sci. (2022) 32:41. 
    
 \bibitem{L-W-2014} F. Lin; C. Wang,
Recent developments of analysis for hydrodynamic flow of nematic liquid crystals,
Philos. Trans. R. Soc. Lond. Ser. A Math. Phys. Eng. Sci. 372 (2014), no. 2029, 20130361, 18 pp. 

\bibitem{lion} P. Lions,
Mathematical topics in fluid mechanics. Vol. 2. Compressible models. Oxford Lecture Series in Mathematics and its Applications, 10. Oxford Science Publications. The Clarendon Press, Oxford University Press, New York, 1998.

\bibitem{live1} T. B. Liverpool; M. C. Marchetti, Rheology of Active Filament Solutions, Phys. Rev. Lett. 97 (2006), 268101.
 

 \bibitem{mjr} M. C.  Marchetti; J.  F. Joanny; S. Ramaswamy, T. B.  Liverpool;
J. Prost; M. Rao; R. Aditi Simha, Hydrodynamics of soft active matter, Rev. Mod. Phys. 85 (2013), 1143.


\bibitem{mpm} P.  B. Mucha;  M. Pokorn\'y;  E. Zatorska, Existence of stationary weak solutions for compressible heat conducting flows. Handbook of Mathematical Analysis in Mechanics of Viscous Fluids, 2595-2662, Springer, Cham, 2018.

\bibitem{P-K-O-2004} W. F. Paxton; K. C. Kistler; C. C. Olmeda; A. Sen; S. K. St. Angelo;
Y. Cao; T. E. Mallouk; P. E. Lammert; V. H. Crespi,
{Catalytic nanomotors: autonomous movement of striped nanorods}, J. Amer. Chem. Soc. 126 (41) (2014),  13424--13431.

\bibitem{QW2021}
Z. Qiu; Y. Wang,  
Martingale solution for stochastic active liquid crystal system, 
Discrete Contin. Dyn. Syst. 41 (2021), no. 5, 2227--2268. 

 \bibitem{novo}
A. Novotn\'y;  I. Stra\v skraba,  {Introduction to the Mathematical Theory of Compressible Flow.}
Oxford Lecture Series in Mathematics and its Applications, 27. Oxford University Press, Oxford, 2004.

 \bibitem{novo1}
S. Novo;  A. Novotn\'y,  On the existence of weak solutions to the steady compressible Navier-Stokes equations when the density is not square integrable,  J. Math. Fluid Mech. {42} 3  (2002),  531-550.

\bibitem{plot} P.  Plotnikov;  J. Sokolowski,  {Concentrations of solutions to time-discretized compressible Navier-Stokes equations}, Comm. Math. Phys.   {258}(3)  (2005),  567-608.

 \bibitem{pw} P.  Plotnikov; W. Weigant,   Steady 3D viscous compressible flows with adiabatic exponent $\g\in (1,\infty)$,   J. Math. Pures Appl.  104,(2015) 58-82.

\bibitem{R-Y-2013} M. Ravnik; J. M. Yeomans,
{Confined active nematic flow in cylindrical capillaries}, Phys. Rev. Lett. 110 (2013), 026001.

 \bibitem{sai} D. Saintillan;  M.   Shelley, Instabilities and pattern formation in active particle suspensions:
Kinetic theory and continuum simulations, Phys. Rev. Lett. 100 (2008), 178103.

\bibitem{stein} E. Stein, { Singular integrals and differentiability properties of functions,} Princeton Univ. Press, Princeton, New Jersey, 1970.

\bibitem{XuX2022}
X. Xu,  
Recent analytic development of the dynamic Q-tensor theory for nematic liquid crystals,
Electron. Res. Arch. 30 (2022), no. 6, 2220-2246. 

    \bibitem{zie} W. P.  Ziemer, Weakly Differentiable Functions. Springer, New York, 1989.

\end {thebibliography}
\end{document}